\documentclass[a4paper,11pt]{article}

\title{The infinite simple group~$V$ of Richard~J. Thompson:
  presentations by permutations}
\author{Collin Bleak \& Martyn Quick\\[10pt]
  School of Mathematics \& Statistics, University of St Andrews,\\
  St Andrews, Fife, KY16 9SS, United Kingdom\\
  \texttt{cb211@st-andrews.ac.uk}, \texttt{mq3@st-andrews.ac.uk}}

\usepackage{amsmath,amssymb}
\usepackage[margin=2.5cm]{geometry}
\usepackage{theorem}
\usepackage{tikz}
\usepackage[arrow]{xy}
\usepackage{pb-diagram,pb-xy}

\allowdisplaybreaks

\numberwithin{equation}{section}

\newcommand{\A}{\mathcal{A}}
\newcommand{\C}{\mathfrak{C}}

\newcommand{\eqstate}[1]{Equation~(\ref{#1})}
\newcommand{\GAP}{\textsf{GAP}}
\newcommand{\length}[1]{\mathopen{|}#1\mathclose{|}}
\newcommand{\nbd}{\nobreakdash-}
\newcommand{\R}{$\mathcal{R}$}
\newcommand{\set}[2]{\{\,#1\mid#2\,\}}
\newcommand{\spc}{\vspace{\baselineskip}}
\newcommand{\swap}[2]{(#1\;#2)}

\newcommand{\Sw}[1]{\mathcal{T}_{#1}}
\newcommand{\supp}{\operatorname{supp}}

\renewcommand{\epsilon}{\varepsilon}
\renewcommand{\geq}{\geqslant}
\renewcommand{\leq}{\leqslant}
\renewcommand{\emptyset}{\varnothing}

\newcommand{\qed}{\hspace*{\fill}$\square$}

\theorembodyfont{\slshape}
\newtheorem{thm}{Theorem}[section]
\newtheorem{lemma}[thm]{Lemma}
\newtheorem{cor}[thm]{Corollary}
\newtheorem{prop}[thm]{Proposition}

\newenvironment{proof}{%
  \begin{trivlist}\item\textsc{Proof:}}{\qed\end{trivlist}}

\renewcommand{\theenumi}{(\roman{enumi})}
\renewcommand{\labelenumi}{{\normalfont\theenumi}}

\begin{document}

\maketitle

\begin{abstract}
  We show that one can naturally describe elements of R.~Thompson's
  finitely presented infinite simple group~$V$, known by Thompson to
  have a presentation with four generators and fourteen relations, as
  products of permutations analogous to transpositions.  This
  perspective provides an intuitive explanation towards the simplicity
  of~$V$ and also perhaps indicates a reason as to why it was one of
  the first discovered infinite finitely presented simple groups: it
  is (in some basic sense) a relative of the finite alternating
  groups.  We find a natural infinite presentation for~$V$ as a group
  generated by these ``transpositions,'' which presentation bears
  comparison with Dehornoy's infinite presentation and which enables
  us to develop two small presentations for~$V$: a human-interpretable
  presentation with three generators and eight relations, and a
  Tietze-derived presentation with two generators and seven
  relations.

\end{abstract}



\section{Introduction}

In this article, we investigate R.~Thompson's group~$V$ from a
mostly-unexplored perspective.  As a consequence we derive new, and
hopefully elegant, presentations of this well-known group and
introduce a simple and dextrous notation for handling computations in
$V$.

Recall that the group~$V$ first appears in Thompson's 1965
notes~\cite{ThompsonNotes} and is given there as one of two
``first-examples'' of infinite finitely presented simple groups (along
with its simple subgroup~$T$, called~``\textbf{C}'' in those notes).
Since then, it has been the focus of a large amount of subsequent
research (see, for example, \cite{BMM, BleakSal, Brin-higher,
  Dehornoy, Higman, Lawson-completions, Thumann} for a small part of
that research).  Thompson's group~$V$ arises in various other
settings, for example, Birget~\cite{Birget-circuits, Birget2}
investigates connections to circuits and complexity while
Lawson~\cite{Lawson2} considers links to inverse monoids and \'{e}tale
groupoids.

We shall demonstrate that one can consider~$V$ as a symmetric group
acting, not on a finite set, but instead on a Cantor algebra (the
algebra of basic clopen sets in a Cantor space).  We focus upon
certain well-known properties of a finite symmetric group, namely
being generated by transpositions and being transitive in its natural
action.  Reflecting these two fundamental properties, a finite
symmetric group possesses a Coxeter-type presentation, with generating
set~$\mathcal{T}$ corresponding to a set of appropriate transpositions
and relations $t^{2} = 1$ for all $t \in \mathcal{T}$, \ $(tu)^{2} =
1$ when $t,u \in \mathcal{T}$ correspond to transpositions of disjoint
support and $(tu)^{3} = 1$ when $t \neq u$ but the corresponding
transpositions have intersecting support.  If we exploit the fact that
these generators have order~$2$, this third type of relation can be
rewritten as $t^{-1}ut = u^{-1}tu$ and indeed in the symmetric group
this conjugate equals another transposition~$v$, namely that whose
support satisfies $\supp v = (\supp t)u$.

In the context of a Cantor algebra, the analogues of transpositions
are piecewise affine maps which ``swap'' a pair of basic open sets.
We shall observe that Thompson's group~$V$ is generated by such
transpositions of the standard Cantor algebra and hence derive an
infinite Coxeter-like presentation for~$V$, as appears in
Theorem~\ref{thm:infpres} below.

As is well known, the standard Cantor algebra admits a natural
tree-structure where the nodes correspond to the basic open sets in
Cantor space~$\C$ and these nodes are indexed by finite words in the
alphabet $X = \{0,1\}$. Consequently, we label our transpositions by
two incomparable words $\alpha$~and~$\beta$ from~$X^{\ast}$.  Indeed,
for such $\alpha$~and~$\beta$, we write~$t_{\alpha,\beta}$ for the
element of~$V$ that is the transposition defined in
Equation~\eqref{eq:swapmap}.  We shall also write
$s_{\alpha,\beta}$~and~$\swap{\alpha}{\beta}$ for symbols representing
elements in two abstract groups whose presentations we give in
Theorems~\ref{thm:infpres} and~\ref{thm:main}, respectively.  We
specifically use different notations for each of these elements so as
to distinguish between the elements of each abstract group and the
actual transformations of Cantor space.  The thrust of our work is to
demonstrate that the two abstract groups are isomorphic to~$V$ and
that under the isomorphisms these three elements
$s_{\alpha,\beta}$,~$\swap{\alpha}{\beta}$ and~$t_{\alpha,\beta}$
correspond.  The first two of our families of relations appearing in
Theorem~\ref{thm:infpres} reflect that these~$t_{\alpha,\beta}$ act as
transpositions so have order~$2$, commute when their supports are
disjoint, and conjugate in a manner analogous to transpositions in
symmetric groups when their supports intersect appropriately.

Passing from the setting of actions of finite permutation groups on
finite sets to the setting of corresponding actions of infinite groups
on Cantor algebras has further implications for the resulting
presentation.  Namely, due to the self-similar nature of Cantor space,
each generating transposition can be factorised.  To be precise,
each~$t_{\alpha,\beta}$ satisfies what we call a split relation:
$t_{\alpha,\beta} = t_{\alpha0,\beta0} \, t_{\alpha1,\beta1}$.  This
provides the third family of relations that are seen in
Theorem~\ref{thm:infpres}.  They have the consequence that not only is
every element of~$V$ a product of our
transpositions~$t_{\alpha,\beta}$, but also we can re-express any such
product as one that involves an even number of transpositions.  Thus
one can simultaneously view R.~Thompson's group~$V$ as an infinite
analogue of both the finite alternating groups and of the finite
symmetric groups.

It follows quite easily from the presentation in
Theorem~\ref{thm:infpres} that any transposition~$t_{\gamma,\delta}$
can be obtained by conjugation using only those~$t_{\alpha,\beta}$
with $\length{\alpha},\length{\beta} \leq 3$, for example, and this
motivates an effort to find a finite presentation involving
permutations and their relations, where these permutations involve
only the nodes in the first three levels of the tree.
Theorem~\ref{thm:main} provides this presentation (involving three
generators and eight relations).  Note here that we depart slightly
from the Coxeter-style of presentation: we exploit the presence of the
symmetric group of degree~$4$ acting upon~$X^{2}$ to reduce further
the presentation, at the cost of employing a ``three-cycle'' as a
generator.  Of note, this human-interpretable presentation is much
smaller than the currently known finite presentation for~$V$ (given by
Thompson~\cite{ThompsonNotes} and discussed in detail in Cannon, Floyd
and Parry's survey~\cite{CFP}), which has four generators and fourteen
relations.

As a technical exercise, we further reduce the presentation in
Theorem~\ref{thm:main} to a $2$\nbd generator and $7$\nbd relation
presentation, found in Theorem~\ref{thm:2gen-KB}.  The resulting
presentation is small, but not so readily interpretable by humans.

Our infinite presentation in Theorem~\ref{thm:infpres} bears
comparison with Dehornoy's infinite presentation for~$V$ (see
\cite[Proposition~3.20]{Dehornoy}).  Dehornoy's presentation
highlights different aspects as to why the group~$V$ can be considered
as a fundamental object in group theory, and even in mathematics,
bearing out, as it does, the connection of~$V$ to systems with
equivalences under associativity and commutativity.  

Our viewpoint of $V$ as a form of a symmetric or alternating group
perhaps hints at why $V$~arose as one of first two known examples of
an infinite simple finitely presented group.  Permuting sets is a
basic activity, and the Cantor algebra represents a fundamental way to
pass from a finite to an infinite context, thus it seems natural that
researchers eventually noticed~$V$.  To give further background, we
note there are many generalisations of~$V$ to infinite simple finitely
presented groups, all of which owe their simplicity to the same
fundamental idea (similar to the reason why the alternating groups are
simple).  One such family is the Higman--Thompson groups~$G_{n,r}$ for
which $V = G_{2,1}$, see~\cite{Higman}.   (The group~$G_{n,r}$ is
simple for $n$~even.  When $n$~is odd, one must pass to the commutator
subgroup of index~$2$, reflecting the observation that the
corresponding split relations in~$G_{n,r}$ do not change the parity of
any decomposition as a product of transpositions.)  Other families
include the Brin--Thompson groups~$nV$ for which $V = 1V$,
see~\cite{Brin-higher}, and the groups~$nV_{m,r}$ that generalise the
previous two families, see~\cite{MPN}, and where we have similar
simplicity considerations, see~\cite{Brin-simplicity}.  The finite
presentability of these groups comes from the much stronger fact that
they are all in fact $\mathrm{F}_{\infty}$ groups.  (There is a
beautiful argument of the $\mathrm{F}_{\infty}$ nature of these groups
given in~\cite{Thumann}, which applies to many of these ``relatives''
of~$V$.  In many specific cases, $\mathrm{F}_{\infty}$ arguments
already exist for individual groups and for classes of groups in these
families.  See, for example, \cite{BelkForrest,Brown,BrownGeog,KMPN}.)
The ideas of this paper ought to apply to all of these groups of
``Thompson type'' in aiding in the discovery of natural and small
presentations.  On the other hand, the infinite family of
finitely-presented infinite simple groups arising from the
Burger-Mozes construction and following related work (see, e.g.,
\cite{BurgMozes97,BurgMozes01a,Rattaggi}) are of an entirely different
nature, and the methods employed here do not seem appropriate to that
context.

We mention here a debt to Matatyu Rubin and Matthew~G. Brin.  Rubin
indicated to Brin a proof of the simplicity of~$V$, which uses the
generation of~$V$ by transpositions with restricted support on Cantor
space.  Brin set this proof out briefly in his
paper~\cite{Brin-higher} and developed the ideas to extend the proof
to the groups $nV$, which he carries out in the short
paper~\cite{Brin-simplicity}.  It is not a stretch to say that the
current article would not exist without that thread of previous
research.

\subsection{A note on content}

The first two sections of this article are intended for the interested
mathematician and provide structure and insights into these sorts of
groups.  The outline of the proofs of the theorems are found towards
the end of Section~\ref{sec:prelims}.  The sections that follow are
more technical and verify the details required for those proofs.

\subsection{Statement of results and some notation}

Let $X = \{ 0,1 \}$.  We write~$X^{\ast}$ for all finite sequences
$x_{1}x_{2}\dots x_{k}$ where $k \geq 0$ and each $x_{i} \in X$.  In
particular, we assume that $X^{\ast}$~contains the empty
word~$\epsilon$.  We view the elements of~$X^{\ast}$ as representing
the nodes on the infinite binary rooted tree with edges between nodes
if they are represented by words which differ by a suffix of
length~$1$.  (Figure~\ref{fig:T2} illustrates this tree together with
the nodes labelled by elements of~$X^{\ast}$.)  Similarly, we give the
standard definition of the Cantor set~$\C$ as~$X^{\omega}$, the set of
all infinite sequences $x_{1}x_{2}x_{3}\dots$ of elements of~$X$ under
the product topology (starting with $X$~endowed with the discrete
topology).  Thus points in~$\C$ correspond to boundary points of the
infinite binary rooted tree.

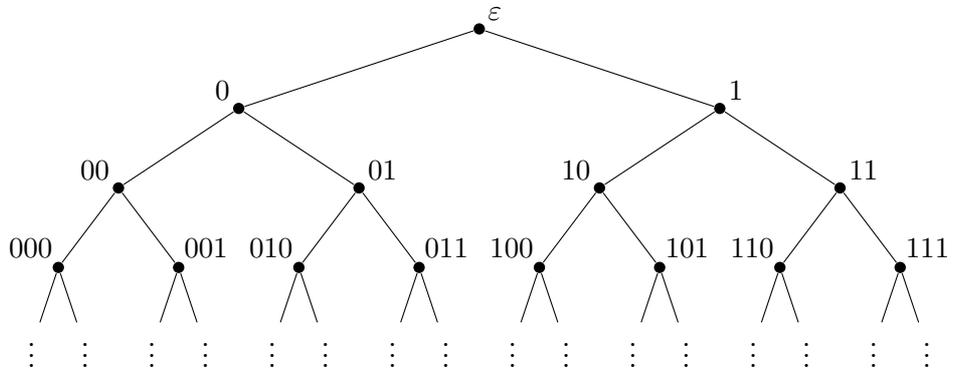
\begin{figure}
  \begin{center}
    \begin{tikzpicture}[
        inner sep=1.5pt,
        level distance=30pt,
        level 1/.style={sibling distance=180pt},
        level 2/.style={sibling distance=90pt},
        level 3/.style={sibling distance=45pt},
        level 4/.style={sibling distance=20pt}
      ]
      \node (root) [circle,fill,label={45:$\epsilon$}] {}
      child {node (0) [circle,fill,label={135:$0$}] {}
        child {node (00) [circle,fill,label={135:$00$}] {}
          child {node (000) [circle,fill,label={100:$000$}] {}
            child {node {$\vdots$}}
            child {node {$\vdots$}}
          }
          child {node (001) [circle,fill,label={80:$001$}] {}
            child {node {$\vdots$}}
            child {node {$\vdots$}}
          }
        }
        child {node (01) [circle,fill,label={45:$01$}] {}
          child {node (010) [circle,fill,label={100:$010$}] {}
            child {node {$\vdots$}}
            child {node {$\vdots$}}
          }
          child {node (011) [circle,fill,label={80:$011$}] {}
            child {node {$\vdots$}}
            child {node {$\vdots$}}
          }
        }
      }
      child {node (1) [circle,fill,label={45:$1$}] {}
        child {node (10) [circle,fill,label={135:$10$}] {}
          child {node (100) [circle,fill,label={100:$100$}] {}
            child {node {$\vdots$}}
            child {node {$\vdots$}}
          }
          child {node (101) [circle,fill,label={80:$101$}] {}
            child {node {$\vdots$}}
            child {node {$\vdots$}}
          }
        }
        child {node (11) [circle,fill,label={45:$11$}] {}
          child {node (110) [circle,fill,label={100:$110$}] {}
            child {node {$\vdots$}}
            child {node {$\vdots$}}
          }
          child {node (111) [circle,fill,label={80:$111$}] {}
            child {node {$\vdots$}}
            child {node {$\vdots$}}
          }
        }
      }
      ;
    \end{tikzpicture}
  \end{center}
  \caption{The infinite binary rooted tree with nodes labelled by
    elements of~$X^{\ast}$}
  \label{fig:T2}
\end{figure}

If $\alpha \in X^{\ast}$ and $\beta \in X^{\ast} \cup X^{\omega}$,
then we write~$\alpha\beta$ for the concatenation of the two
sequences.  We denote by~$\alpha\C$ the set of elements of~$\C$ with
initial prefix~$\alpha$.  This set is a basic open set in the
topology on~$\C$ and is itself homeomorphic to~$\C$.  We shall
write $\alpha \preceq \beta$ to indicate that $\alpha$~is a prefix
of~$\beta$ (including the possibility that the two sequences are
equal).  This notation then means that $\beta = \alpha\gamma$ for some
$\gamma \in X^{\ast} \cup X^{\omega}$.  Moreover, when $\beta \in
X^{\ast}$, then $\alpha$~and~$\beta$ represent nodes on the infinite
binary rooted tree such that $\beta$~lies on a path descending
from~$\alpha$ (see Figure~\ref{fig:relation}(i)), and therefore
$\beta\C \subseteq \alpha\C$.

\begin{figure}
  \begin{center}
    \begin{tikzpicture}[
        baseline={(root.base)},
        inner sep=0pt,
        level 1/.style={sibling distance=60pt},
        level 2/.style={sibling distance=30pt},
        level distance=20pt
      ]
      \node (root) [circle,fill] {}
      child {node (0) [circle,fill] {}
        child {node (00) {$\vdots$}}
        child {node (01) [circle,fill] {}}
      }
      child {node (1) [circle,fill] {}
        child {node {$\vdots$}}
        child {node {$\vdots$}}
      };
      \fill[gray!30] (-2,-4.5) -- (-2.6,-5.55) -- (-1.4,-5.55) -- cycle;
      \filldraw (-1,-3) circle [radius=2pt]
      (-2,-4.5) circle [radius=2pt]
      ;
      \draw (01) .. controls (0,-2.1) and (-0.8,-2.6) .. (-1,-3);
      \draw (-1,-3) .. controls (-1.5,-3.6) and (-2,-4) .. (-2,-4.5);
      \node (A) at (-0.6,-3) {$\alpha$};
      \node (B) at (-1.6,-4.5) {$\beta$};
      \draw (-2,-4.5) -- (-2.4,-5.2);
      \draw (-2,-4.5) -- (-1.6,-5.2);
      \node (bottom) at (-2,-6) {$\underbrace{\qquad\qquad}_{\beta\C}$};
    \end{tikzpicture}
    \hspace*{50pt}
    \begin{tikzpicture}[
        baseline={(root.base)},
        inner sep=0pt,
        level 1/.style={sibling distance=60pt},
        level 2/.style={sibling distance=30pt},
        level distance=20pt
      ]
      \node (root) [circle,fill] {}
      child {node (0) [circle,fill] {}
        child {node (00) {$\vdots$}}
        child {node (01) [circle,fill] {}}
      }
      child {node (1) [circle,fill] {}
        child {node {$\vdots$}}
        child {node {$\vdots$}}
      };
      \filldraw (-2,-4.5) circle [radius=2pt]
      (0,-5) circle [radius=2pt]
      ;
      \draw (01) .. controls (0,-2.1) and (-0.8,-2.6) .. (-1,-3);
      \draw (-1,-3) .. controls (-1.5,-3.6) and (-2,-4) .. (-2,-4.5);
      \draw (-1,-3) .. controls (0,-3.8) and (-.5,-4.3) .. (0,-5);
      \node (A) at (-1.6,-4.5) {$\alpha$};
      \node (B) at (.4,-5) {$\beta$};
    \end{tikzpicture}
  \end{center}
  \caption{(i)~$\alpha \preceq \beta$ (and the paths representing
    elements of~$\beta\C$); and (ii)~$\alpha\perp\beta$}
  \label{fig:relation}
\end{figure}
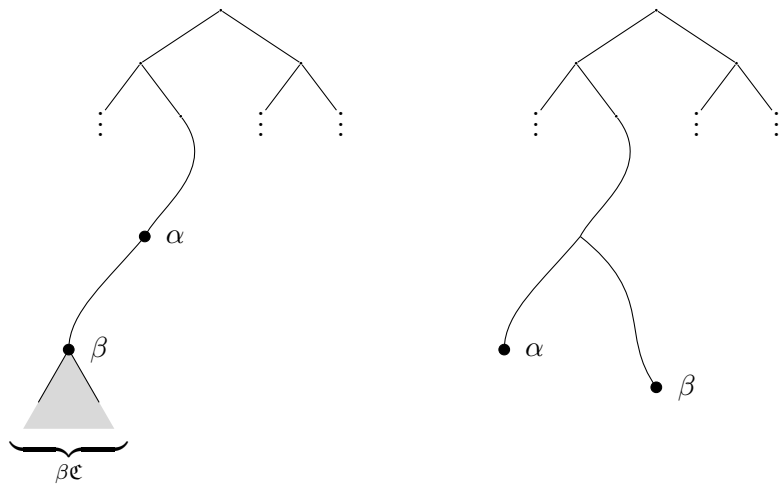

We also write $\alpha \perp \beta$ to denote that both $\alpha
\not\preceq \beta$ and $\beta \not\preceq \alpha$.  Then we shall say
that $\alpha$~and~$\beta$ are \emph{incomparable}.  In this case, the
paths to~$\alpha$ and to~$\beta$ from the root separate at some node
above both $\alpha$~and~$\beta$ (as shown in
Figure~\ref{fig:relation}(ii)), so that $\alpha\C \cap \beta\C =
\emptyset$ (for such $\alpha,\beta \in X^{\ast}$).

If $m$~is a positive integer, we shall also use~$X^{m}$ to denote the
collection of all finite sequences $x_{1}x_{2}\dots x_{m}$ of
\emph{length~$m$} with $x_{i} \in X$ for each~$i$.  We write
$\length{\alpha}$ for the length of the sequence $\alpha \in
X^{\ast}$.

Motivated by the well-known fact (see, for example, \cite{Dehornoy})
that R. Thompson's group $V$ has a partial action on the set of finite
binary rooted trees and equally on the set~$X^{\ast}$, we shall use
the notation~$\gamma \cdot \swap{\alpha}{\beta}$, for
$\alpha,\beta,\gamma \in X^{\ast}$, defined by
\begin{equation}
  \gamma \cdot \swap{\alpha}{\beta} =
  \begin{cases}
    \beta\delta &\text{if $\gamma = \alpha\delta$ for some $\delta \in
      X^{\ast}$}, \\
    \alpha\delta &\text{if $\gamma = \beta\delta$ for some $\delta \in
      X^{\ast}$}, \\
    \gamma &\text{if both $\gamma \perp \alpha$ and $\gamma \perp
      \beta$}, \\
    \text{undefined} &\text{otherwise}.
  \end{cases}
  \label{eq:swap-act}
\end{equation}
Thus $\gamma \cdot \swap{\alpha}{\beta}$ is undefined precisely when
$\gamma \prec \alpha$ or $\gamma \prec \beta$ and when it is defined
it represents a prefix substitution replacing any occurrence
of~$\alpha$ by~$\beta$ and \emph{vice versa}.

The notation appearing in Equation~\eqref{eq:swap-act} above is
motivated via our anticipated isomorphism between~$V$ and the abstract
group defined by the presentation in Theorem~\ref{thm:main}.  The
map~$t_{\alpha,\beta}$ is taken to the element~$\swap{\alpha}{\beta}$
and the above formula reflects the effect of
applying~$t_{\alpha,\beta}$ to a point in the partial action of~$V$
on~$X^{\ast}$.  Note also that we are choosing to write our maps on
the right since in our opinion this enables one to more conveniently
compose a number of maps and such a convention is consistent with
denoting an element of~$V$ by tree-pairs with the domain on the left
and the codomain on the right.  Nevertheless, we still view maps
in~$V$ as being given by prefix substitutions of the infinite
sequences that are elements of the Cantor space so as to be consistent
with other work on these groups.

Our results are as follows:

\begin{thm}
  \label{thm:infpres}
  Let $\A$~to be the set of all symbols~$s_{\alpha,\beta}$ where
  $\alpha,\beta \in X^{\ast}$ with $\alpha \perp \beta$.  Then
  R.~Thompson's group~$V$ has an infinite presentation with generating
  set~$\A$ and relations
  \begin{equation}
  \begin{aligned}
    s_{\alpha,\beta}^{2} &= 1 \\*
    s_{\gamma,\delta}^{-1} \; s_{\alpha,\beta} \; s_{\gamma,\delta} &=
    s_{\alpha\cdot\swap{\gamma}{\delta},
      \beta\cdot\swap{\gamma}{\delta}} \\*
    s_{\alpha,\beta} &= s_{\alpha0,\beta0} \, s_{\alpha1,\beta1}
  \end{aligned}
  \label{eq:inf-rels}
  \end{equation}
  where $\alpha$,~$\beta$, $\gamma$ and~$\delta$ range over all
  sequences in~$X^{\ast}$ such that $\alpha \perp \beta$, \ $\gamma
  \perp \delta$, and $\alpha \cdot \swap{\gamma}{\delta}$ and $\beta
  \cdot \swap{\gamma}{\delta}$ are defined.
\end{thm}

Our primary finite presentation for the group~$V$ has three generators
$a$,~$b$ and~$c$, but, as mentioned above, it is most naturally
expressed in terms of a ``permutation-like'' notation extending the
transpositions in the infinite presentation.  Our generators $a$,~$b$
and~$c$ then correspond to permutations that we denote
$\swap{00}{01}$, \ $( 01 \; 10 \; 11 )$ and~$\swap{1}{00}$,
respectively, and the relations are similarly expressed in terms of
elements~$\swap{\alpha}{\beta}$ that we define fully in
Section~\ref{sec:prelims}.  This ``human-readable'' presentation is as
follows:

\begin{thm}
  \label{thm:main}
  R.~Thompson's group~$V$ has a finite presentation with three
  generators $\swap{00}{01}$, \ $(01 \; 10 \; 11)$ and~$\swap{1}{00}$
  and eight relations
\begin{enumerate}
  \renewcommand{\theenumi}{\R\arabic{enumi}}
  \renewcommand{\labelenumi}{\theenumi.}

\item
  \label{R:S4}
  $\swap{00}{01}^{2} = (01 \; 10 \; 11)^{3} =
  \bigl( \swap{00}{01} \, (01 \; 10 \; 11) \bigr)^{4} = 1$;

\item
  \label{R:12}
  $\swap{1}{01}^{\swap{1}{00}} = \swap{00}{01}$;

\item
  \label{R:split}
  $\swap{1}{00} = \swap{10}{000} \, \swap{11}{001}$;

\item
  \label{R:[23,23]}
  $[ \swap{00}{010} , \swap{10}{111} ] =
  [ \swap{00}{011} , \swap{10}{111} ] = 1$;

\item
  \label{R:[33,23]}
  $[ \swap{000}{010} , \swap{10}{110} ] = 1$.

\end{enumerate}  
\end{thm}

We shall provide words in terms of the generators $a$,~$b$ and~$c$ to
express these relations later in Equation~\eqref{eq:R-words}.
Observe that the
Relations~\ref{R:S4} tells us that $\swap{00}{01}$ and $(01 \;
10 \; 11)$ satisfy the relations of the symmetric group~$S_{4}$, so
the subgroup that they generate is isomorphic to some quotient
of~$S_{4}$.  In fact, it will turn out that this subgroup \emph{is}
isomorphic to~$S_{4}$.

The element~$\swap{\alpha}{\beta}$ will correspond to the element of
Thompson's group $V$ that maps a point of the Cantor set that has
prefix~$\alpha$ to a point with prefix~$\beta$ and \emph{vice versa}.
Relations~\ref{R:[23,23]} and~\ref{R:[33,23]}
then reflect the fact that certain elements of~$V$ commute because
they have disjoint support.

\begin{figure}[b]
  \[
  \begin{tikzpicture}[
      inner sep=0pt,
      baseline=-30pt,
      level distance=20pt,
      level 1/.style={sibling distance=30pt},
      level 2/.style={sibling distance=15pt}
    ]
    \node (root) [circle,fill] {}
    child {node (0) [circle,fill] {}
      child {node (00) {$1$}}
      child {node (01) {$2$}}}
    child {node (1) [circle,fill] {}
      child {node (10) {$3$}}
      child {node (11) [circle,fill] {}
        child {node (110) {$4$}}
        child {node (111) {$5$}}}};
  \end{tikzpicture}
  \xrightarrow{\;\;u\;\;}
  \begin{tikzpicture}[
      inner sep=0pt,
      baseline=-30pt,
      level distance=20pt,
      level 1/.style={sibling distance=30pt},
      level 2/.style={sibling distance=15pt}
    ]
    \node (root) [circle,fill] {}
    child {node (0) [circle,fill] {}
      child {node (00) {$2$}}
      child {node (01) {$1$}}}
    child {node (1) [circle,fill] {}
      child {node (10) {$5$}}
      child {node (11) [circle,fill] {}
        child {node (110) {$3$}}
        child {node (111) {$4$}}}};
  \end{tikzpicture}
  \qquad\qquad\qquad
  \begin{tikzpicture}[
      inner sep=0pt,
      baseline=-30pt,
      level distance=20pt,
      level 1/.style={sibling distance=30pt},
      level 2/.style={sibling distance=15pt}
    ]
    \node (root) [circle,fill] {}
    child {node (0) [circle,fill] {}
      child {node (00) {$1$}}
      child {node (01) {$2$}}}
    child {node (1) [circle,fill] {}
      child {node (10) {$3$}}
      child {node (11) {$4$}}};
  \end{tikzpicture}
  \xrightarrow{\;\;v\;\;}
  \begin{tikzpicture}[
      inner sep=0pt,
      baseline=-30pt,
      level distance=20pt,
      level 1/.style={sibling distance=30pt},
      level 2/.style={sibling distance=15pt}
    ]
    \node (root) [circle,fill] {}
    child {node (0) [circle,fill] {}
      child {node (00) {$1$}}
      child {node (01) {$4$}}}
    child {node (1) [circle,fill] {}
      child {node (10) {$2$}}
      child {node (11) {$3$}}};
  \end{tikzpicture}  
  \]
  \caption{Two elements given by tree-pairs that generate for
    R.~Thompson's group~$V$}
  \label{fig:2gens}
\end{figure}
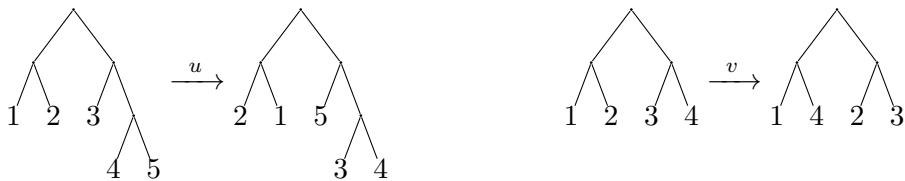

We show that $V$~is generated by the two elements $u$~and~$v$
described by the tree-pairs in Figure~\ref{fig:2gens}.  Transforming
the presentation in Theorem~\ref{thm:main} to one using these two
generators via Tietze transformations (as described in
Corollary~\ref{cor:2gen-Tietze}) and reducing the nine resulting
relations using the Knuth--Bendix algorithm, as implemented in the
\GAP\ package KBMAG~\cite{GAP,KBMAG}, results in the surprising two
generator and seven relation presentation given below:

\begin{thm}
  \label{thm:2gen-KB}
  R.~Thompson's group~$V$ has a finite presentation with two
  generators $u$~and~$v$ and the seven relators
  \begin{align*}
  &u^6, \;\; v^3, \;\; (u^{3}v)^4, \\
  &v^{-1}u(u^{2}v^{-1})^{2}u^{3}vu^{-1}v^{-1}u^{3}vu(uvu^{2}(uv^{-1}u^{3}v)^{3})^{2}uv^{-1}u^{3}v^{-1}, \\
  &uv^{-1}u^{3}v^{-1}u^{-2}v^{-1}uvu^{2}v^{-1}u^{-1}vu^{2}v^{-1}uvu^{-1}(u^{-1}v^{-1})^{2}u^{3}vu^{-1}, \\
  &v(uv^{-1}u^{3}v^{-1})^{2}u^{-1}v^{-1}u^{3}v^{-1}u^{-1}v^{-1}u^{3}v, \\
    &uvu^{3}vuv^{-1}u^{-2}v^{-1}u(u^{2}v)^{2}(u^{2}v^{-1})^{2}u^{3}vu^{-2}v^{-1}u^{3}v
    .
  \end{align*}
\end{thm}

This reduction to seven relations caught the authors by surprise,
but perhaps it is not so unexpected in view of the deficiency (as
defined, for example, in~\cite[\S14.1]{Robinson}) of the presentations
in Theorems~\ref{thm:main} and~\ref{thm:2gen-KB} both being~$-5$.

\section{Further preliminaries and the proofs of the main theorems}
\label{sec:prelims}

This section contains the heart of the mathematics within the article.
We present all the remaining preliminaries required to fully
understand the statements of the theorems listed in the introduction,
in particular, unpacking the presentations that we use.  We then
provide the proofs, subject to deferring technical calculations to the
sections that follow.

\subsection{\boldmath R.~Thompson's group~$V$}

One view of Thompson's group~$V$ is as a group of certain
homeomorphisms of the Cantor set~$\C$, namely those that are finite
products of the elements~$t_{\alpha,\beta}$, for $\alpha,\beta \in
X^{\ast}$ with $\alpha \perp \beta$, defined as follows
\begin{equation}
  x t_{\alpha,\beta} = \begin{cases}
    \beta y &\text{if $x = \alpha y$ for some $y \in \C$}; \\
    \alpha y &\text{if $x = \beta y$ for some $y \in \C$}; \\
    x &\text{otherwise}
  \end{cases}
  \label{eq:swapmap}
\end{equation}
(see Brin~\cite[Lemma~12.2]{Brin-higher}).  Note that the
map~$t_{\alpha,\beta}$ has the effect of swapping those elements
of~$\C$ that have an initial prefix~$\alpha$ with those that have an
initial prefix~$\beta$ and fixing all other points in~$\C$.  A general
element of~$V$ is often denoted by a pair of finite binary rooted
trees representing the domain and codomain of the map.  We label the
leaves of these two trees by the numbers $1$,~$2$, \dots,~$n$ (for
some~$n$) and this then specifies that our element of~$V$ has the
effect of substituting the prefix from~$X^{\ast}$ corresponding to the
node in the domain tree labelled~$i$ by the member of~$X^{\ast}$
corresponding to the node in the codomain tree with the same label
(for each~$i$).  For example, Figure~\ref{fig:t-example} provides such
tree-pairs for the map~$t_{100,11}$ as just defined.

\begin{figure}
  \[
  \begin{tikzpicture}[
      inner sep=0pt,
      baseline=-30pt,
      level distance=20pt
    ]
    \node (root) [circle,fill] {}
    child {node (0) {$1$}}
    child {node (1) [circle,fill] {}
      child {node (10) [circle,fill] {}
        child {node (100) {$2$}}
        child {node (101) {$3$}}}
      child {node (11) {$4$}}};
  \end{tikzpicture}
  \quad
  \xrightarrow{\qquad}%{\swap{100}{11}}
  \quad
  \begin{tikzpicture}[
      inner sep=0pt,
      baseline=-30pt,
      level distance=20pt
    ]
    \node (root) [circle,fill] {}
    child {node (0) {$1$}}
    child {node (1) [circle,fill] {}
      child {node (10) [circle,fill] {}
        child {node (100) {$4$}}
        child {node (101) {$3$}}}
      child {node (11) {$2$}}};
  \end{tikzpicture}
  \]
  \caption{The map~$t_{100,11}$ denoted using tree-pairs}
  \label{fig:t-example}
\end{figure}
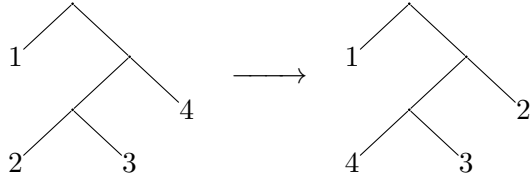

From the definition, it is visible that $t_{\alpha,\beta}^{2} = 1$.
Equations of this type (as $\alpha$~and~$\beta$ range over all
incomparable pairs from~$X^{\ast}$) will form our family of
\emph{order relations}.

If we shift our attention to conjugation, it is a straightforward
calculation in~$V$, along the lines of the familiar one concerning
conjugation of permutations demonstrated to undergraduates in a first
course on group theory, to verify that
\[
t_{\gamma,\delta}^{-1} \, t_{\alpha,\beta} \, t_{\gamma,\delta} =
t_{\alpha\cdot\swap{\gamma}{\delta}, \beta\cdot\swap{\gamma}{\delta}}
\]
whenever $\alpha\cdot\swap{\gamma}{\delta}$ and
$\beta\cdot\swap{\gamma}{\delta}$ are both defined.  We call this
resultant family of relations in~$V$ our \emph{conjugacy relations}.
At this point, we also note that we will use an exponential notation
for conjugation, so the left-hand side of the above relation will also
be denoted by~${t_{\alpha,\beta}}^{t_{\gamma,\delta}}$ in what follows.

Our final family of relations exploit the action of our
elements~$t_{\alpha,\beta}$ on basic open sets of the Cantor set~$\C$
and the self-similar structure of this space.  Specifically, if
$\alpha \in X^{\ast}$, then the basic open set~$\alpha\C$ splits into
two subsets, namely the set of all elements of~$\C$ with initial
prefix~$\alpha0$ and those with initial prefix~$\alpha1$.  In view of
this, we obtain the equation
\begin{equation}
  t_{\alpha,\beta} = t_{\alpha0,\beta0} \, t_{\alpha1,\beta1}
  \label{eq:t-split}
\end{equation}
when $\alpha,\beta \in X^{\ast}$ with $\alpha \perp \beta$.  We refer
to the family of these relations as \emph{split relations}.

\subsection{\boldmath Deriving the presentations for~$V$}

One of the presentations that we use in this article is that found by
R. Thompson and discussed in Cannon--Floyd--Parry
(see~\cite[Lemma~6.1]{CFP}).  This presentation has four generators
$A$,~$B$, $C$ and~$\pi_{0}$ and fourteen relations.  We state these
relations when we need them at the start of Section~\ref{sec:CFP}
below.

As described in Theorem~\ref{thm:infpres}, the first of our new
presentations involves the order relations, conjugacy relations and
split relations of~$V$ just described.  To be precise, we define
$P_{\infty}$~to be the group having the infinite presentation with
generating set $\A = \set{s_{\alpha,\beta}}{\alpha,\beta \in X^{\ast},
  \; \alpha \perp \beta}$ and the relations listed in
Equation~\eqref{eq:inf-rels}.  Of course, we already know that that
$V$~is generated by the maps~$t_{\alpha,\beta}$ and that these satisfy
the order, conjugacy and split relations.  This ensures that there is
a surjective homomorphism $\phi \colon P_{\infty} \to V$ given by
$s_{\alpha,\beta} \mapsto t_{\alpha,\beta}$ for $\alpha,\beta \in
X^{\ast}$ with $\alpha \perp \beta$.  When establishing
Theorem~\ref{thm:infpres}, we shall be observing that $\phi$~is
actually an isomorphism.

\spc

We now describe our primary finite presentation for~$V$, which has
three generators $a$,~$b$ and~$c$ and eight relations but, more
importantly, can be readily understood by a human.  The majority of
our calculations will take place with this presentation and so, as
indicated above, we develop a helpful notation that is parallel to
that used in finite permutation groups.  It is a consequence of
Higman~\cite{Higman} that all the relations that hold in~$V$ can be
detected as consequences of products using tree-pairs of some bounded
size.  This motivates our presentation which employs essentially a
finite subcollection of the order, conjugacy and split relations
from~$P_{\infty}$ and involving only some swaps~$\swap{\alpha}{\beta}$
all satisfying $\length{\alpha},\length{\beta} \leq 3$.  (To aid
reducing the number of relations required we encode a copy of the
symmetric group of degree~$4$, corresponding to acting on~$X^{2}$,
within our presentation.)

Accordingly, the three generators $a$,~$b$ and~$c$ of the
group~$P_{3}$ that we define will represent cyclic permutations of
basic open sets.  As stated above, we shall write $\swap{00}{01}$,
$(01 \; 10 \; 11)$ and~$\swap{1}{00}$ for $a$,~$b$ and~$c$,
respectively.  This reflects the fact that, under the isomorphism that
we shall establish between $P_{3}$~and~$V$, the element~$a$
corresponds to the map~$t_{00,01}$ that interchanges the basic open
sets $00\C$ and~$01\C$, \ $b$~corresponds to the product~$t_{01,10}\,
t_{01,11}$ (inducing a $3$\nbd cycle of the sets $01\C$,~$10\C$
and~$11\C$), and $c$~corresponds to~$t_{1,00}$.  We extend this
notation by defining elements~$\swap{\alpha}{\beta}$ that we will
refer to as \emph{swaps} below and a formula for each in terms of
$a$,~$b$ and~$c$ will be extracted from the definitions we make.  The
element~$\swap{\alpha}{\beta}$ will correspond to~$t_{\alpha,\beta}$
under the isomorphism.  It is these swaps that appear in our list of
relations found in Theorem~\ref{thm:main} above.

Once we have defined the swaps below, we can translate the
Relations~\ref{R:S4}--\ref{R:[33,23]} into words in $a$,~$b$ and~$c$.
The result is the following restatement of Theorem~\ref{thm:main}:

\begin{thm}
  \label{thm:mainWithWords}
  R.~Thompson's group~$V$ has a finite presentation with three
  generators $a$,~$b$ and~$c$ and the following eight relations:
\begin{equation}
  \begin{aligned}
    \begin{gathered}
      a^{2} = b^{3} = (ab)^{4} = 1, \\
      c = a^{bcacaa^{ba}} \, a^{b^{-1}cacaa^{b^{-1}a}}, \\
      [ a^{b^{-1}cac} , a^{b^{-1}caca^{b}a^{b^{-1}a}} ] = 1,
    \end{gathered}
    \qquad\qquad
    &\begin{gathered}
       c^{ac} = a, \\
       [ a^{bcac} , a^{b^{-1}caca^{b}a^{b^{-1}a}} ] = 1, \\
       [ a^{bca^{bca}} , a^{bcaca^{b}a^{b^{-1}a}} ] = 1.
     \end{gathered}
  \end{aligned}
  \label{eq:R-words}
\end{equation}
\end{thm}

The careful reader will doubtless have observed that the eighth
relation can be shortened by conjugating by~$a$.  The presentation as
listed is a direct consequence of the interpretation of
Theorem~\ref{thm:main} in terms of $a$,~$b$ and~$c$.  No effort has
been made in its statement to reduce the length of the relations.

Having found this nice presentation for~$V$, we felt obligated to
reduce the relations employing the available technology, specifically the
Knuth--Bendix Algorithm.  This algorithm shortens the above relations,
although the results are no longer particularly transparent.  To carry
out these reductions, we used the implementation of the algorithm
found in the freely available KBMAG package~\cite{KBMAG} in
\GAP~\cite{GAP} in the following way.  Denote the eight relators
corresponding to the equations in~\eqref{eq:R-words} by
$r_{1}$,~$r_{2}$, \dots,~$r_{8}$.  We can construct a rewriting system
associated to each of the groups $Q_{i} = \langle \, a,b,c \mid
r_{1},r_{2},\dots,r_{i-1},r_{i+1},\dots,r_{8} \, \rangle$ for $i =
4$,~$5$, \dots,~$8$ in sequence.  The systems that KBMAG constructs
are not confluent, but nevertheless enable us to replace each~$r_{i}$
by a Tietze-equivalent (in the group~$Q_{i}$) shorter relation.  This
process is repeated until the resulting relations stabilise.  As a
consequence, the normal closure, in the free group on~$\{a,b,c\}$, of
the following eight relations is identical to that of our original
list:
\begin{equation}
\begin{gathered}
  a^{2} = b^{3} = (ab)^{4} = 1, \\
  c^{-1}(ac)^{2}a = 1, \\
  (cab^{-1}aba)^{2}cb(cabab^{-1}a)^{2} = 1, \\
  a(cb)^{2}a(b^{-1}c)^{2}bcabcb^{-1}cab^{-1}acb^{-1}(cb)^{2}ab^{-1} =
  1, \\
  ab^{-1}cbc(ab^{-1})^{2}cbcb^{-1}a(b^{-1}c)^{2}babcb^{-1}cab^{-1} =
  1, \\
  ca(b^{-1}c)^{2}bacabacbc(b^{-1}ca)^{2}b(cb^{-1})^{2}(acb)^{2}cb^{-1}cab^{-1}
  = 1.
\end{gathered}
\label{eq:KB-words}
\end{equation}
We observe this mechanical process produces considerably shorter
relations than our original eight in Equation~\eqref{eq:R-words}.

The presentation in Theorem~\ref{thm:2gen-KB} is deduced in a manner
that similarly depends upon the use of the Knuth--Bendix Algorithm.
One first applies Tietze transformations to pass to a $2$\nbd
generator presentation employing generators $u$~and~$v$ and relations
deduced from the list~\eqref{eq:KB-words}.  We shall describe these
Tietze transformations by expressing $u$~and~$v$ in terms of $a$,~$b$
and~$c$ (adding extraneous generators), and expressing $a$,~$b$
and~$c$ in terms of $u$~and~$v$ (removing extraneous generators).  The
relevant formulae are
\[
u = a(a^ba^{b^{-1}})^{caca^{b}a^{b^{-1}a}}
\qquad \text{and}\qquad
v = b
\]
and
\[
a = u^{3}, \qquad
b = v, \qquad
c = (u^{3})^{vu^{-2}vu^{3}} \, (u^{3})^{vu^{-1}vu^{3}v}.
\]
These formulae can be deduced by direct calculation in~$V$.  This
application of Tietze transformations is expanded upon a little in
Section~\ref{sec:finaldetails} and Corollary~\ref{cor:2gen-Tietze}
provides the intermediate step to the theorem.  (This corollary is
established by purely theoretical means and does not rely upon
computer calculation.)

We then employ the same relation reduction process using KBMAG as
described earlier and this shows that two of the nine relations
resulting from the Tietze transformations are extraneous.  In this
manner we have deduced Theorem~\ref{thm:2gen-KB} from
Theorem~\ref{thm:main}.

\spc

We now proceed to formally define the swaps~$\swap{\alpha}{\beta}$ for
$\alpha,\beta \in X^{\ast}$ with $\alpha \perp \beta$ and
$\length{\alpha},\length{\beta} \leq 3$ in terms of our generators
$a$,~$b$ and~$c$ in order to present the group~$P_{3}$.  To start off,
we define swaps~$\swap{\alpha}{\beta}$ for $\alpha,\beta \in X^{2}$ as
follows:
\begin{equation}
  \begin{aligned}
    \swap{00}{01} &= a,
    &\swap{00}{10} &= a^{b},
    &\swap{00}{11} &= a^{b^{-1}} \\
    \swap{01}{10} &= a^{ba},
    &\swap{01}{11} &= a^{b^{-1}a},
    &\swap{10}{11} &= a^{bab}
  \end{aligned}
  \label{D:22}
\end{equation}
Here, and in all that follows, we shall also adopt the convention that
the swap~$\swap{\beta}{\alpha}$ coincides with~$\swap{\alpha}{\beta}$
whenever the latter has already been defined.  We write $\Sw{2}$~for
the set of swaps~$\swap{\alpha}{\beta}$ with $\alpha,\beta \in X^{2}$.

The swaps in~$\Sw{2}$ and their effect when conjugating will be of
fundamental importance in our calculations.  Accordingly we spend a
little time expanding upon the above definitions before we define the
remaining swaps.  In the Relations~\ref{R:S4}, we have assumed that
$a$~and~$b$ satisfy the relations of the symmetric group~$S_{4}$ and
the formulae on the right-hand side of Equation~\eqref{D:22} are those
that correspond to transpositions in~$S_{4}$.  Consequently, when we
multiply and conjugate elements of~$\Sw{2}$, they behave in exactly
the same way as transpositions do.  In particular, we can view
individual elements of~$\Sw{2}$, and, by extension, products of such
swaps, as transformations of the set~$X^{2}$.  Indeed, our
notation~$\gamma \cdot \swap{\alpha}{\beta}$, as defined in
Equation~\eqref{eq:swap-act}, when $\alpha,\beta,\gamma \in X^{2}$, is
precisely the formulae for these maps.  Our assumption of
Relations~\ref{R:S4} justifies our using products of swaps
from~$\Sw{2}$ as maps $X^{2} \to X^{2}$, as we shall do explicitly,
for example, in Lemma~\ref{lem:23^2} below.  Similarly, we have
written $( 01 \; 10 \; 11 )$ for the generator~$b$,
since it follows from the Relations~\ref{R:S4} that $b$~is equal to
the product~$\swap{01}{10} \, \swap{01}{11}$, which induces a $3$\nbd
cycle on~$X^{2}$.

To define the remaining swaps~$\swap{\alpha}{\beta}$, where
$\length{\alpha},\length{\beta} \leq 3$, we need one further piece of
notation.  If $x \in X$, we define~$\bar{x}$ to be the other element
in~$X$; that is,
\[
\bar{x} = \begin{cases}
  1 &\text{if $x = 0$}, \\
  0 &\text{if $x = 1$}.
\end{cases}
\]
Then for any $x,y,z \in X$, we make our definitions in the following
order:
\begin{gather}
  \swap{0}{1} = \swap{00}{10} \, \swap{01}{11};
  \label{D:11} \\[10pt]
  \begin{aligned}
    \swap{1}{00} &= c,
    &\swap{1}{01} &= \swap{1}{00}^{\swap{00}{01}}, \\
    \swap{0}{1x} &= \swap{1}{0x}^{\swap{0}{1}}; \qquad\qquad
  \end{aligned} \label{D:12} \\[10pt]
  \begin{aligned}
    \swap{1}{00x} &= \swap{00}{1x}^{\swap{1}{00}}, \qquad\qquad
    &\swap{1}{01x} &= \swap{1}{00x}^{\swap{00}{01}}, \\
    \swap{0}{1xy} &= \swap{1}{0xy}^{\swap{0}{1}};
  \end{aligned} \label{D:13} \\[10pt]
  \begin{aligned}
    \swap{00}{01x} &= \swap{1}{01x}^{\swap{1}{00}},
    &\swap{01}{00x} &= \swap{00}{01x}^{\swap{00}{01}}, \\
    \swap{1x}{1\bar{x}y} &= \swap{0x}{0\bar{x}y}^{\swap{0}{1}},
    &\swap{1x}{0yz} &= \swap{0\bar{y}}{0yz}^{\swap{0\bar{y}}{1x}}, \\
    \swap{0x}{1yz} &= \swap{1x}{0yz}^{\swap{0}{1}}; \qquad\qquad
  \end{aligned} \label{D:23} \\[10pt]
  \begin{aligned}
    \swap{000}{001} &= \swap{1}{000}^{\swap{1}{001}},
    &\swap{000}{010} &= \swap{1}{000}^{\swap{1}{010}}, \\
    \swap{000}{011} &= \swap{1}{000}^{\swap{1}{011}},
    &\swap{001}{011} &= \swap{1}{001}^{\swap{1}{011}}, \\
    \swap{xy0}{xy1} &= \swap{000}{001}^{\swap{00}{xy}}. \qquad\qquad
  \end{aligned} \label{D:33}
\end{gather}
Finally, for distinct $\kappa,\lambda \in X^{2}$, fix a
product~$\rho_{\kappa\lambda}$ of swaps from~$\Sw{2}$ that moves $00$
to~$\kappa$ and $01$ to~$\lambda$ when viewed, as described above, as
a map $X^{2} \to X^{2}$.  Define
\begin{equation}
  \swap{\kappa x}{\lambda y} = \swap{00x}{01y}^{\rho_{\kappa\lambda}}
  \label{D:other33}
\end{equation}
for $(x,y) \in \{ (0,0), (0,1), (1,1) \}$.  In this way, we have now
defined all swaps~$\swap{\alpha}{\beta}$ where $\length{\alpha},
\length{\beta} \leq 3$.

Having made these definitions, it is a straightforward matter to
convert the relations \ref{R:S4}--\ref{R:[33,23]} into the
list~\eqref{eq:R-words} of actual words expressed in the generators
$a$,~$b$ and~$c$, completing the translation of Theorem~\ref{thm:main}
into Theorem~\ref{thm:mainWithWords}.

\subsection{Proofs of the main theorems}

We now provide the proofs of the main theorems (that is,
Theorems~\ref{thm:infpres} and~\ref{thm:main}), subject to information
that we shall establish in the sections of the paper that follow.
Here we link the groups $V$,~$P_{3}$ and~$P_{\infty}$.  Specifically,
we build homomorphisms between these groups as indicated in the
following diagram:

\begin{equation}
  \setlength{\dgARROWLENGTH}{1.2em}
  \begin{diagram}
    \node[2]{P_{3}} \arrow{e,tb,A,V}{\tau}{\text{Tietze}}
    \node{\bar{P}_{3}} \arrow{se,t,A}{\theta} \\
    \node{\langle \bar{A},\bar{B},\bar{C},\bar{\pi}_{0} \rangle}
    \arrow{ne,t,J}{i_{0}} \node[3]{\langle s_{00,01} , s_{01,10} ,
      s_{10,11} , s_{1,00} \rangle} \arrow{sw,t,L}{i_{1}} \\
    \node[2]{V} \arrow{nw,t,A}{\psi} \node{P_{\infty}}
    \arrow{w,t,A}{\phi}
  \end{diagram}
  \label{eq:hexagon}
\end{equation}

We already know that $\phi \colon P_{\infty} \to V$ is a surjective
homomorphism.  We now describe the other parts of the hexagon of maps.

The majority of the work in
Sections~\ref{sec:level3}--\ref{sec:finaldetails} involves the
presentation for the group~$P_{3}$, where we establish information
about the swaps~$\swap{\alpha}{\beta}$ defined above.  In
Section~\ref{sec:level3}, we  verify that these swaps satisfy the
order relations, conjugacy relations and split relations of
R.~Thompson's group~$V$, providing we restrict to those involving only
swaps~$\swap{\alpha}{\beta}$ with $\length{\alpha},\length{\beta} \leq
3$.  Then in Section~\ref{sec:CFP}, we establish that four specific
elements $\bar{A}$,~$\bar{B}$, $\bar{C}$ and~$\bar{\pi}_{0}$
in~$P_{3}$ satisfy the fourteen relations listed in
Cannon--Floyd--Parry~\cite{CFP} that define the group~$V$.  In that
section, we ensure that we only rely upon the consequences of our
Relations~\ref{R:S4}--\ref{R:[33,23]} established in
Section~\ref{sec:level3}.  This then guarantees the existence of a
surjective homomorphism $\psi \colon V \to \langle \bar{A}, \bar{B},
\bar{C}, \bar{\pi}_{0} \rangle$.

In the final section, we establish that $\langle
\bar{A},\bar{B},\bar{C},\bar{\pi}_{0} \rangle$~coincides with the
group~$P_{3}$ (see Proposition~\ref{prop:P3-gens}), from which it
follows that the natural inclusion map $i_{0}$~is also surjective.

Amongst the generators for~$P_{3}$ are the elements $a =
\swap{00}{01}$ and $b = (01 \; 10 \; 11)$, which satisfy the relations
of the symmetric group~$S_{4}$ (Relations~\ref{R:S4}).  Of course,
$S_{4}$~also enjoys a presentation in terms of transpositions
involving only order and conjugacy relations.  Accordingly, we apply
Tietze transformations to convert the presentation for~$P_{3}$ into
one for a group~$\bar{P}_{3}$ with generators
$\swap{00}{01}$,~$\swap{01}{10}$, $\swap{10}{11}$ and~$\swap{1}{00}$
and some order, conjugacy and split relations (specifically
translations of \ref{R:12}--\ref{R:[33,23]}, together with the new
ones to replace~\ref{R:S4}).  Thus we have on the one hand, an
isomorphism $\tau \colon P_{3} \to \bar{P}_{3}$ and, on the other, a
surjective homomorphism $\theta \colon \bar{P}_{3} \to \langle
s_{00,01}, s_{01,10}, s_{10,11}, s_{1,00} \rangle$ (a subgroup
of~$P_{\infty}$), since the relations defining~$\bar{P}_{3}$ all hold
in~$P_{\infty}$.

We can now deduce that $P_{3} = \langle
\bar{A},\bar{B},\bar{C},\bar{\pi}_{0} \rangle$ is not trivial, since
successively composing the appropriate maps in~\eqref{eq:hexagon}
sends, for example, $a = \swap{00}{01}$ to the non-identity
element~$t_{00,01}$ in~$V$.  Hence, from simplicity of~$V$, we
conclude $P_{3} \cong V$ and therefore, subject to the work in
Sections~\ref{sec:level3}--\ref{sec:finaldetails}, establish
Theorem~\ref{thm:main}.

Finally, it is relatively straightforward to observe that if
$\alpha$~is a non-empty sequence in~$X^{\ast}$, then there is a
product~$w$ involving only the swaps $\swap{00}{01}$,~$\swap{01}{10}$,
$\swap{10}{11}$ and~$\swap{1}{00}$ such that $00 \cdot w = \alpha$.
From this, one quickly deduces, principally relying upon the conjugacy
relations, that one can conjugate~$s_{00,01}$ by some element
of~$\langle s_{00,01}, s_{01,10}, s_{10,11}, s_{1,00} \rangle$ to any
generator~$s_{\alpha,\beta}$ where $(\alpha,\beta) \neq (0,1)$.  This
ensures that the inclusion~$i_{1}$ is surjective.  Hence $P_{\infty}
\cong V$ also and we have established Theorem~\ref{thm:infpres}.

The remaining sections are perhaps technical, but carry out the
deferred work just as described above.  We hope that these sections
will quickly impress the reader with the utility of the permutation
notation~$\swap{\alpha}{\beta}$ in performing calculations within
R.~Thompson's group~$V$.

\paragraph{Remarks}
\begin{enumerate}
\item If our goal had been simply to establish
  Theorem~\ref{thm:infpres}, then our work in the next two sections
  would have been greatly reduced.  Indeed, it is actually rather easy
  to show that the group~$P_{\infty}$ defined by all the relations
  listed in Theorem~\ref{thm:infpres} is isomorphic to~$V$ (for
  example, the relations that Dehornoy~\cite{Dehornoy} lists are
  particularly straightforward to deduce).  However, from the
  viewpoint of the high transitivity of the action of~$V$ on the
  Cantor set~$\C$, one expects that it would be enough to restrict to
  relations involving swaps~$\swap{\alpha}{\beta}$ with
  $\length{\alpha}$~and~$\length{\beta}$ bounded.  Indeed, this is
  what leads to our presentation for the group~$P_{3}$ and the small
  set of relations, \ref{R:S4}--\ref{R:[33,23]}, where we are relying
  upon a small subset involving only swaps~$\swap{\alpha}{\beta}$ with
  $\length{\alpha},\length{\beta} \leq 3$.  Establishing
  that this set of relations is sufficient is the delicate business of
  Sections~\ref{sec:level3} and~\ref{sec:CFP}.

\item Now that we have established that the groups
  $P_{3}$,~$P_{\infty}$ and~$V$ are all isomorphic, it is safe to use
  the notation~$\swap{\alpha}{\beta}$ as a convenient notation for the
  map~$t_{\alpha,\beta}$ as defined earlier.  We can then perform
  computations within R.~Thompson's group~$V$ employing this notation,
  for example, along the lines of those in the sections that follow,
  and we hope this will be of use to those working with elements in
  this group.
\end{enumerate}

\section{\boldmath Verification of relations to level~$3$}
\label{sec:level3}

Our principal aim is to establish that all the relations holding in
R.~Thompson's group~$V$ can be deduced from
Relations~\ref{R:S4}--\ref{R:[33,23]}.  In this section, we complete
the first stage of our technical calculations by establishing
essentially a subset of the infinitely many relations in the
list~\eqref{eq:inf-rels}, namely those order relations, the conjugacy
relations and the split relations involving only
swaps~$\swap{\alpha}{\beta}$ with $\length{\alpha},\length{\beta} \leq
3$.  It will turn out that these are enough to then deduce the
fourteen relations for~$V$ found in~\cite{CFP} as we shall see in
Section~\ref{sec:CFP}.

Accordingly, in this section, we shall verify all relations of the
form
\begin{align*}
  \swap{\alpha}{\beta}^{2} &= 1 \\
  \swap{\alpha}{\beta}^{\swap{\gamma}{\delta}} &=
  \swap{\alpha\!\cdot\!\swap{\gamma}{\delta}\;}{\beta\!\cdot\!\swap{\gamma}{\delta}} \\
  \swap{\alpha}{\beta} &= \swap{\alpha0}{\beta0} \,
  \swap{\alpha1}{\beta1}
\end{align*}
whenever any swap~$\swap{\kappa}{\lambda}$ appearing above satisfies
$\length{\kappa},\length{\lambda} \leq 3$.  (So, for example, the
split relation $\swap{\alpha}{\beta} = \swap{\alpha0}{\beta0} \,
\swap{\alpha1}{\beta1}$ needs only to be verified for
$\length{\alpha},\length{\beta} \leq 2$ in order that the swaps on the
right-hand side of the equation fulfil this requirement.)  The main
focus throughout the section will be in establishing all the conjugacy
relations~$\sigma^{\tau} = \upsilon$ and we will consider these
relations when we select $\sigma$,~$\tau$ and~$\upsilon$ from the
following sets:
\begin{align*}
  \Sw{2} &= \set{\swap{\alpha}{\beta}}{\text{$(\alpha,\beta) = (0,1)$
      or $\alpha,\beta \in X^{2}$}} \\
  \Sw{3} &= \set{\swap{\alpha}{\beta}}{\alpha, \beta \in X^{3}} \\
  \Sw{mn} &= \set{\swap{\alpha}{\beta}}{\alpha \in X^{m}, \; \beta \in
    X^{n}}, \qquad \text{where $1 \leq m < n \leq 3$}
\end{align*}

We start our verification by first noting that $\swap{1}{00}$~is a
conjugate of~$\swap{00}{01}$ by Relation~\ref{R:12} and hence has
order dividing~$2$ by the first relation in~\ref{R:S4}.  From this and
the definitions of the swaps, we now know that
$\swap{\alpha}{\beta}^{2} = 1$ whenever $\alpha \perp \beta$ with
$\length{\alpha},\length{\beta} \leq 3$.  We will use this fact
throughout the rest of this section.

We step through the various relations, essentially introducing
``longer'' swaps through the stages.  Accordingly, we start with
relations involving only swaps from $\Sw{12} \cup \Sw{2}$, then
introducing swaps from~$\Sw{13}$, and so on.  We need to take various
side-trips from this general direction in order to successfully
establish all the relations we want.

\subsection{\boldmath Relations involving $\Sw{12}$ and~$\Sw{2}$ only}

With careful analysis, one can soon establish which conjugacy
relations involve only swaps selected from~$\Sw{12} \cup \Sw{2}$.
These are, for $x,y \in X$, the following three relations:
\begin{align}
  \swap{x}{\bar{x}y}^{\swap{0}{1}} &= \swap{\bar{x}}{xy}
  \label{eq:12^11}
  \\
  \swap{x}{\bar{x}y}^{\swap{\bar{x}y}{\bar{x}\bar{y}}} &=
  \swap{x}{\bar{x}\bar{y}}
  \label{eq:12^22}
  \\
  \swap{x}{\bar{x}{y}}^{\swap{x}{\bar{x}\bar{y}}} &=
  \swap{\bar{x}y}{\bar{x}\bar{y}}
  \label{eq:12^12=22}
\end{align}
These are actually very easy to verify.  For
example, \eqstate{eq:12^11} simply follows from our definition
of~$\swap{0}{1x}$ (for $x \in X$) in~\eqref{D:12}, while we can
establish \eqstate{eq:12^22} using our definition of~$\swap{1}{01}$
in~\eqref{D:12} and then conjugating, if necessary, by~$\swap{0}{1}$
and using the now established Equation~\eqref{eq:12^11}.  We now have
the first step in our verification and results along the lines of the
following lemma will occur throughout our progress.

\begin{lemma}
  \label{lem:12^2}
  All conjugacy relations of the form $\sigma^{\tau} = \upsilon$ where
  $\sigma,\upsilon \in \Sw{12}$ and $\tau \in \Sw{2}$ can be deduced
  from Relations~\ref{R:S4}--\ref{R:[33,23]}.\qed
\end{lemma}

This lemma contributes now to establishing
\eqstate{eq:12^12=22}, since we observe that, for the case $x = 1$
and~$y = 0$, the equation is Relation~\ref{R:12} and that the general
equation can then be deduced by conjugating by a product of elements
from~$\Sw{2}$.  Specifically, conjugating by~$\swap{0}{1}$ gives
$\swap{0}{11}^{\swap{0}{10}} = \swap{10}{11}$ and then subsequently
the two equations now established by
$\swap{00}{01}$~and~$\swap{10}{11}$, respectively, establishes the
final cases.

\subsection{\boldmath Relations involving $\Sw{12}$, $\Sw{13}$
  and~$\Sw{2}$ only}

When we introduce swaps from~$\Sw{13}$, in addition to those from
$\Sw{12}$ and~$\Sw{2}$, the relations that need to be verified are,
for $x,y,z \in X$, the following:
\begin{align}
  \swap{x}{\bar{x}yz}^{\swap{0}{1}} &= \swap{\bar{x}}{xyz}
  \label{eq:13^11}
  \\
  \swap{x}{\bar{x}yz}^{\swap{\bar{x}y}{\bar{x}\bar{y}}} &=
  \swap{x}{\bar{x}\bar{y}z}
  \label{eq:13^22}
  \\
  \swap{x}{\bar{x}yz}^{\swap{x}{\bar{x}y}} &= \swap{xz}{\bar{x}y}
  \label{eq:13^12=22}
\end{align}
Both \textbf{Equations~(\ref{eq:13^11})} and~\textbf{(\ref{eq:13^22})}
follow quickly from the definitions in \eqref{D:13}.  They establish:

\begin{lemma}
  \label{lem:13^2}
  All conjugacy relations of the form $\sigma^{\tau} = \upsilon$ where
  $\sigma,\upsilon \in \Sw{13}$ and $\tau \in \Sw{2}$ can be deduced
  from Relations~\ref{R:S4}--\ref{R:[33,23]}.\qed
\end{lemma}

If we expand the definition of~$\swap{1}{00z}$ from
Equation~\eqref{D:13}, we obtain $\swap{1}{00z}^{\swap{1}{00}} =
\swap{00}{1z}$ for $z \in X$.  Now conjugate by an appropriate product
of elements from~$\Sw{2}$, using Lemmas~\ref{lem:12^2}
and~\ref{lem:13^2} similarly to the argument used for
Equation~\eqref{eq:12^12=22}, to establish \eqstate{eq:13^12=22}.

\subsection{\boldmath First batch of relations involving $\Sw{2}$ and
  $\Sw{23}$ only}

We now turn to relations involving swaps from~$\Sw{23}$.  There are
additional relations that we shall establish later involving only
swaps from $\Sw{2}$~and~$\Sw{23}$, but for now we are concerned with
the following equations (in particular, the first of our split
relations) for $x,y \in X$, for \emph{distinct} $\kappa,\lambda,\mu
\in X^{2}$, and for any $\tau \in \Sw{2}$ for which the first is
defined:
\begin{align}
  \swap{\kappa}{\lambda x}^{\tau} &=
  \swap{\kappa\!\cdot\!\tau\;}{(\lambda x)\!\cdot\!\tau}
  \label{eq:23^2}
  \\
  \swap{x}{\bar{x}y} &= \swap{x0}{\bar{x}y0} \, \swap{x1}{\bar{x}y1}
  \label{eq:12-split}
  \\
  \swap{\kappa}{\lambda x}^{\swap{\mu}{\lambda x}} &=
  \swap{\kappa}{\mu}
  \label{eq:23^23=22}
\end{align}

We shall need the following lemma to be able to manipulate the initial
swaps from~$\Sw{23}$ from which the others are built, as given
in~\eqref{D:23}.  We shall then establish \eqstate{eq:23^2} in the
form of Lemma~\ref{lem:23^2} below.

\begin{lemma}
  \label{lem:23help}
  \begin{enumerate}
  \item
    \label{23split-commute}
    $\swap{1}{00}$,~$\swap{10}{000}$ and~$\swap{11}{001}$ all commute;
  \item
    \label{(01,00x)-shrink}
    $\swap{01}{00x}^{\swap{1}{00}} = \swap{01}{1x}$;
  \end{enumerate}
\end{lemma}

\begin{proof}
  \ref{23split-commute} Using Relation~\ref{R:split} and the fact
  that all the swaps involved have order dividing~$2$, we conclude
  that $\swap{10}{000}$~and~$\swap{11}{001}$ commute.  It then
  follows, again using~\ref{R:split}, that $\swap{1}{00}$~also
  commutes with these elements.

  \ref{(01,00x)-shrink}~Pull apart the formula for~$\swap{01}{00x}$
  using the definitions in~\eqref{D:23}, and then apply
  Relation~\ref{R:12} and Equation~\eqref{eq:13^12=22} as follows:
  \[
  \swap{01}{00x}^{\swap{1}{00}} =
  \swap{1}{01x}^{\swap{1}{00} \, \swap{00}{01} \, \swap{1}{00}} =
  \swap{1}{01x}^{\swap{1}{01}} = \swap{01}{1x}.
  \]
\end{proof}

\begin{lemma}
  \label{lem:23^2}
  All conjugacy relations of the form $\sigma^{\tau} = \upsilon$ where
  $\sigma,\upsilon \in \Sw{23}$ and $\tau \in \Sw{2}$ can be deduced
  from Relations~\ref{R:S4}--\ref{R:[33,23]}.
\end{lemma}

\begin{proof}
  The first half of the proof deals with the case when $\sigma =
  \swap{\kappa}{\lambda0}$ for distinct $\kappa,\lambda \in X^{2}$.
  To establish this, we must first verify that the
  swap~$\swap{00}{010}$ commutes with~$\swap{10}{11}$.  However, to
  achieve this, we actually work first with the
  swap~$\swap{10}{000}$.  Indeed,
  \begin{align*}
    \swap{10}{000} \, \swap{01}{11} &=
    \swap{1}{00} \, \swap{11}{001} \, \swap{01}{11}
    &&\text{by Rel.~\ref{R:split}} \\
    &= \swap{1}{00} \, \swap{01}{11} \, \swap{01}{001}
    &&\text{by the def.\ of~$\swap{11}{001}$ in~\eqref{D:23}} \\
    &= \swap{01}{001} \, \swap{01}{11} \, \swap{1}{00}
    &&\text{by Lem.~\ref{lem:23help}\ref{(01,00x)-shrink} twice} \\
    &= \swap{01}{11} \, \swap{11}{001} \, \swap{1}{00}
    &&\text{by the def.\ in~\eqref{D:23}} \\
    &= \swap{01}{11} \, \swap{10}{000}
    &&\text{by Lem.~\ref{lem:23help}\ref{23split-commute} and
      Rel.~\ref{R:split}.}
  \end{align*}
  As the definition of~$\swap{10}{000}$ in Equation~\eqref{D:23} is
  $\swap{00}{010}^{\swap{00}{01} \, \swap{01}{10}}$, we conclude that
  $\swap{00}{010}$ commutes with $\swap{01}{11}^{\swap{01}{10} \,
    \swap{00}{01}} = \swap{10}{11}$.

  So we now turn to the required relation when $\sigma =
  \swap{\kappa}{\lambda0}$ for distinct $\kappa,\lambda \in X^{2}$.
  As described earlier, we view~$\tau$ as a permutation of~$X^{2}$.
  As such a map, suppose $\tau$~maps $\kappa$~and~$\lambda$ to
  $\mu$~and~$\nu$, respectively.  Then by the definition in
  Equation~\eqref{D:23} there are products $\rho_{1}$~and~$\rho_{2}$
  of elements from~$\Sw{2}$ such that
  \[
  \swap{\kappa}{\lambda0} = \swap{00}{010}^{\rho_{1}}
  \qquad \text{and} \qquad
  \swap{\mu}{\nu0} = \swap{00}{010}^{\rho_{2}}.
  \]
  Specifically, $\rho_{1}$~and~$\rho_{2}$ are products that, when
  viewed as permutations of~$X^{2}$, map $00$~and~$01$ to
  $\kappa$~and~$\lambda$ and to $\mu$~and~$\nu$, respectively.  Then
  $\rho_{1} \tau \rho_{2}^{-1} \in \langle \swap{10}{11} \rangle$
  since it fixes both $00$~and~$01$.  Hence, by our previous
  calcuation, $\swap{00}{010}$~commutes with~$\rho_{1} \tau
  \rho_{2}^{-1}$, so
  \begin{equation}
    \swap{\kappa}{\lambda0}^{\tau} = \swap{00}{010}^{\rho_{1}\tau} =
    \swap{00}{010}^{\rho_{2}} = \swap{\mu}{\nu0}.
    \label{eq:23^2-case0}
  \end{equation}
  Thus, we have established Equation~\eqref{eq:23^2} in the case when
  $x = 0$.

  To deduce the equation for $x = 1$, we proceed similarly and first
  need to establish that $\swap{00}{011}$~commutes
  with~$\swap{10}{11}$.  We calculate as follows:
  \begin{align*}
    \swap{10}{000}^{\swap{01}{000}}
    &= \swap{10}{000}^{\swap{1}{00} \, \swap{01}{10} \, \swap{1}{00}}
    &&\text{by Lem.~\ref{lem:23help}\ref{(01,00x)-shrink}} \\
    &= \swap{10}{000}^{\swap{01}{10} \, \swap{1}{00}}
    &&\text{by Lem.~\ref{lem:23help}\ref{23split-commute}} \\
    &= \swap{01}{000}^{\swap{1}{00}}
    &&\text{by Eq.~\eqref{eq:23^2-case0}} \\
    &= \swap{01}{10}
    &&\text{by Lem.~\ref{lem:23help}\ref{(01,00x)-shrink}}.
  \end{align*}
  Conjugate by~$\swap{01}{10}$ and use Equation~\eqref{eq:23^2-case0}
  again to establish the formula $\swap{01}{000}^{\swap{10}{000}} =
  \swap{01}{10}$.  Now we find
  \begin{align*}
    \swap{11}{001} \, \swap{01}{10} &=
    \swap{10}{000} \, \swap{1}{00} \, \swap{01}{10}
    &&\text{by Rel.~\ref{R:split}} \\
    &= \swap{10}{000} \, \swap{01}{000} \, \swap{1}{00}
    &&\text{by Lem.~\ref{lem:23help}\ref{(01,00x)-shrink}} \\
    &= \swap{01}{10} \, \swap{10}{000} \, \swap{1}{00}
    &&\text{as just established} \\
    &= \swap{01}{10} \, \swap{11}{001}
    &&\text{by Rel.~\ref{R:split}.}
  \end{align*}
  Thus $[ \swap{11}{001} , \swap{01}{10} ] = 1$, and then using the
  definition of~$\swap{11}{001}$ in~\eqref{D:23} and the
  Relations~\ref{R:S4}, we deduce $[ \swap{00}{011} , \swap{10}{11} ]
  = 1$.  We then proceed as in the first half of the proof, using this
  new equation, and we establish Equation~\eqref{eq:23^2} when $x =
  1$, completing the proof of the lemma.
\end{proof}

\eqstate{eq:12-split} now follows by conjugating
Relation~\ref{R:split} by an appropriate product of elements
from~$\Sw{2}$ using Lemmas~\ref{lem:12^2} and~\ref{lem:23^2}.

The establishment of \eqstate{eq:23^23=22} requires an intermediate
observation first.  We use Lemma~\ref{lem:23^2} to tell us that
$\swap{11}{001}$~commutes with~$\swap{01}{10}$, so
\begin{align*}
  \swap{01}{10} \, \swap{11}{001} &=
  \swap{11}{001} \, \swap{01}{10} \\
  &= \swap{10}{000} \, \swap{1}{00} \, \swap{01}{10}
  &&\text{by Rel.~\ref{R:split}} \\
  &= \swap{10}{000} \, \swap{01}{000} \, \swap{1}{00}
  &&\text{by Lem.~\ref{lem:23help}\ref{(01,00x)-shrink}} \\
  &= \swap{10}{000} \, \swap{01}{000} \, \swap{10}{000} \,
  \swap{11}{001}
  &&\text{by Rel.~\ref{R:split}.}
\end{align*}
Hence $\swap{01}{000}^{\swap{10}{000}} = \swap{01}{10}$.  By a similar
sequence of calculations we establish
\begin{align*}
  \swap{10}{000} \, \swap{01}{11}
  &= \swap{01}{11} \, \swap{10}{000}
  &&\text{by Lem.~\ref{lem:23^2}} \\
  &= \swap{01}{11} \, \swap{1}{00} \, \swap{11}{001}
  &&\text{by Rel.~\ref{R:split}} \\
  &= \swap{1}{00} \, \swap{01}{001} \, \swap{11}{001}
  &&\text{by Lem.~\ref{lem:23help}\ref{(01,00x)-shrink}} \\
  &= \swap{10}{000} \, \swap{11}{001} \, \swap{01}{001} \,
  \swap{11}{001}
  &&\text{by Rel.~\ref{R:split}},
\end{align*}
so $\swap{01}{001}^{\swap{11}{001}} = \swap{01}{11}$.  We now have two
equations that, once we conjugate by a product of elements
from~$\Sw{2}$, yield the general form of Equation~\eqref{eq:23^23=22}
using Lemma~\ref{lem:23^2} and Relations~\ref{R:S4}.

\subsection{\boldmath Relations involving $\Sw{12}$, $\Sw{2}$
  and~$\Sw{23}$}

The relations that involve swaps from~$\Sw{12}$, $\Sw{23}$
and~$\Sw{2}$ and that definitely include at least one swap from each
of the first two sets are, for $x,y,z,t \in X$, the following:
\begin{align}
  \swap{xy}{\bar{x}z}^{\swap{x}{\bar{x}\bar{z}}} &=
  \swap{\bar{x}z}{\bar{x}\bar{z}y}
  \label{eq:22^12=23}
  \\
  \swap{xy}{\bar{x}zt}^{\swap{x}{\bar{x}z}} &= \swap{xt}{\bar{x}zy}
  \label{eq:23^12=23}
\end{align}
To establish \eqstate{eq:22^12=23}, first calculate
\[
\swap{00}{1x}^{\swap{1}{01} \, \swap{1}{00}} =
\swap{00}{1x}^{\swap{1}{00} \, \swap{00}{01}} =
\swap{1}{00x}^{\swap{00}{01}} =
\swap{1}{01x}
\]
using Relation~\ref{R:12} and the definitions in~\eqref{D:13}.  Hence
\[
\swap{00}{1x}^{\swap{1}{01}} = \swap{1}{01x}^{\swap{1}{00}} = \swap{00}{01x}
\]
by the definitions in~\eqref{D:23}.  This is
Equation~\eqref{eq:22^12=23} in the case when $x = 1$ and $z = 0$.
The general equation then follows by conjugating by a suitable product
of elements from~$\Sw{2}$ and using Lemmas~\ref{lem:12^2}
and~\ref{lem:23^2}.

To establish \eqstate{eq:23^12=23}, we first deal with the case when
$x = 1$ and $z = 0$.  We calculate
\begin{multline*}
  \swap{1y}{00t}^{\swap{1}{00}}
  = \swap{01}{00t}^{\swap{01}{1y} \, \swap{1}{00}}
  = \swap{01}{00t}^{\swap{1}{00} \, \swap{01}{00y}} \\
  = \swap{01}{1t}^{\swap{01}{00y}}
  = \swap{1t}{00y}
\end{multline*}
using the definitions in~\eqref{D:23}, Equation~\eqref{eq:22^12=23}
(twice) and~\eqref{eq:23^23=22}.  Now conjugate by an appropriate
product of elements from~$\Sw{2}$, using Lemmas~\ref{lem:12^2}
and~\ref{lem:23^2}, to conclude that Equation~\eqref{eq:23^12=23}
holds.

\subsection{Intermediate relations}

Our current goal at this stage is to complete the establishment of
relations involving swaps from $\Sw{2}$~and~$\Sw{23}$ to supplement
Equations~\eqref{eq:23^2}--\eqref{eq:23^23=22} already obtained.
However, in order to achieve this, we need some intermediate relations
involving swaps from~$\Sw{3}$ of the form~$\swap{\kappa0}{\kappa1}$
for some $\kappa \in X^{2}$, specifically for $x,y \in X$ and distinct
$\kappa,\lambda \in X^{2}$:
\begin{align}
  \swap{\kappa}{\lambda x}^{\swap{\kappa}{\lambda\bar{x}}} &=
  \swap{\lambda0}{\lambda1}
  \label{eq:23^23=k01}
  \\
  \swap{\kappa}{\lambda x}^{\swap{\lambda0}{\lambda1}} &=
  \swap{\kappa}{\lambda\bar{x}}
  \label{eq:23^k01=23}
  \\
  [ \swap{x}{\bar{x}y} , \swap{\bar{x}\bar{y}0}{\bar{x}\bar{y}1} ] &=
  1
  \label{eq:[12,k01]}
\end{align}

We start by establishing a helpful lemma, analogous to
Lemmas~\ref{lem:23help} and~\ref{lem:23^2}, concerning the swaps of
the form~$\swap{\kappa0}{\kappa1}$.

\begin{lemma}
  \label{lem:k01help}
  \begin{enumerate}
  \item
    \label{(000,001)-shrink}
    $\swap{000}{001}^{\swap{1}{00}} = \swap{10}{11}$;
  \item
    \label{(000,001)-commute}
    $\swap{000}{001}$~commutes with~$\swap{01}{10}$,
    with~$\swap{01}{11}$, and with~$\swap{10}{11}$.
  \item
    \label{k01^2}
    All conjugacy relations of the form $\sigma^{\tau} = \upsilon$,
    where $\sigma,\upsilon \in \Sw{3}$ have the
    form~$\swap{\kappa0}{\kappa1}$ for some $\kappa \in X^{2}$ and
    $\tau \in \Sw{2}$, can be deduced from
    Relations~\ref{R:S4}--\ref{R:[33,23]}.
  \end{enumerate}
\end{lemma}

\begin{proof}
  \ref{(000,001)-shrink}~This follows using the definitions in
  \eqref{D:13}~and~\eqref{D:33}:
  \begin{multline*}
    \swap{000}{001}^{\swap{1}{00}} =
    \swap{1}{000}^{\swap{1}{001} \, \swap{1}{00}} =
    \swap{1}{000}^{\swap{1}{00} \, \swap{00}{11}} \\
    = \swap{00}{10}^{\swap{00}{11}} = \swap{10}{11}.
  \end{multline*}

  \ref{(000,001)-commute}~First we know, by Lemma~\ref{lem:23^2}, that
  $\swap{10}{11}$~commutes with~$\swap{01}{00x}$ for any $x \in X$.
  Conjugating by~$\swap{1}{00}$, using part~\ref{(000,001)-shrink} and
  Lemma~\ref{lem:23help}\ref{(01,00x)-shrink}, establishes that
  $\swap{000}{001}$~commutes with~$\swap{01}{1x}$.

  Then we perform the following calculation:
  \begin{align*}
    \swap{000}{001} \, \swap{10}{11}
    &= \swap{1}{00} \, \swap{10}{11} \, \swap{1}{00} \, \swap{10}{11}
    &&\text{by part~\ref{(000,001)-shrink}} \\
    &= \swap{10}{000} \, \swap{11}{001} \, \swap{10}{11} \,
    \swap{1}{00} \, \swap{10}{11}
    &&\text{by Rel.~\ref{R:split}} \\
    &= \swap{10}{11} \, \swap{11}{000} \, \swap{10}{001} \,
    \swap{1}{00} \, \swap{10}{11}
    &&\text{by Lem.~\ref{lem:23^2}} \\
    &= \swap{10}{11} \, \swap{1}{00} \, \swap{10}{001} \,
    \swap{11}{000} \, \swap{10}{11}
    &&\text{by Eq.~\eqref{eq:23^12=23}} \\
    &= \swap{10}{11} \, \swap{1}{00} \, \swap{10}{11} \,
    \swap{11}{001} \, \swap{10}{000}
    &&\text{by Lem.~\ref{lem:23^2}} \\
    &= \swap{10}{11} \, \swap{1}{00} \, \swap{10}{11} \, \swap{1}{00} \\
    &\mbox{}\hspace*{18em}\makebox[0pt]{by Lem.~\ref{lem:23help}\ref{23split-commute} and
      Rel.~\ref{R:split}} \\
    &= \swap{10}{11} \, \swap{000}{001}
    &&\text{by part~\ref{(000,001)-shrink}.}
  \end{align*}
  Thus $\swap{000}{001}$~commutes with~$\swap{10}{11}$.

  \ref{k01^2}~This follows by the same argument as used in
  Lemma~\ref{lem:23^2}, noting that if a product~$\rho$ of elements
  from~$\Sw{2}$ fixes~$00$ then it lies in the subgroup generated by
  $\swap{01}{10}$,~$\swap{01}{11}$ and~$\swap{10}{11}$ and so commutes
  with~$\swap{000}{001}$ by part~\ref{(000,001)-commute}.
\end{proof}

For \eqstate{eq:23^23=k01}, start with the equations
$\swap{01}{1x}^{\swap{01}{1\bar{x}}} = \swap{10}{11}$ and then
conjugate by~$\swap{1}{00}$.  Use
Lemmas~\ref{lem:23help}\ref{(01,00x)-shrink}
and~\ref{lem:k01help}\ref{(000,001)-shrink} to conclude
\[
\swap{01}{00x}^{\swap{01}{00\bar{x}}} = \swap{000}{001}
\]
for any $x \in X$.  The required equation now follows by conjugating
by a product of elements from~$\Sw{2}$ and using Lemmas~\ref{lem:23^2}
and Lemma~\ref{lem:k01help}\ref{k01^2}.

To establish \eqstate{eq:23^k01=23}, first use Lemma~\ref{lem:23^2} to
conclude that $\swap{10}{000}^{\swap{10}{11}} = \swap{11}{000}$.  Now
conjugate by~$\swap{1}{00}$ and use Equation~\eqref{eq:23^12=23} and
Lemma~\ref{lem:k01help}\ref{(000,001)-shrink} to establish the formula
$\swap{10}{000}^{\swap{000}{001}} = \swap{10}{001}$.  Conjugating by
an appropriate product of elements from~$\Sw{2}$ and use of
Lemmas~\ref{lem:23^2} and~\ref{lem:k01help}\ref{k01^2} establishes
$\swap{\kappa}{\lambda0}^{\swap{\lambda0}{\lambda1}} =
\swap{\kappa}{\lambda1}$, which is sufficient to verify
Equation~\eqref{eq:23^k01=23}.

We deduce \eqstate{eq:[12,k01]} by starting with $[
  \swap{00}{01} , \swap{10}{11} ] = 1$ and conjugating
by~$\swap{1}{00}$ to yield $[ \swap{1}{01} , \swap{000}{001} ] = 1$,
using Relation~\ref{R:12} and
Lemma~\ref{lem:k01help}\ref{(000,001)-shrink}.  Conjugating by an
appropriate product of elements from~$\Sw{2}$ and using
Lemmas~\ref{lem:12^2} and~\ref{lem:k01help}\ref{k01^2} establishes the
required relation.

\subsection{\boldmath Remaining relations involving $\Sw{2}$
  and~$\Sw{23}$}

We now establish all the relations remaining that involve just swaps
from $\Sw{2}$~and~$\Sw{23}$.  Careful analysis shows that the ones we
are currently missing are, for $x,y \in X$ and distinct
$\kappa,\lambda,\mu,\nu \in X^{2}$, the following:
\begin{align}
  [ \swap{\kappa}{\lambda0} , \swap{\mu}{\lambda1} ] &= 1
  \label{eq:[23-k0,23-k1]} \\
  [ \swap{\kappa}{\lambda x} , \swap{\mu}{\nu y} ] &= 1
  \label{eq:[23,23]}
\end{align}
Note that in establishing these equations we are essentially
establishing that swaps from~$\Sw{23}$ that have disjoint support (or,
more accurately, corresponding to maps in~$V$ with disjoint support)
commute.

\eqstate{eq:[23-k0,23-k1]} simply follows from
Lemma~\ref{lem:23help}\ref{23split-commute} using
Lemma~\ref{lem:23^2}.

Use of Lemma~\ref{lem:23^2} deduces $[ \swap{\kappa}{\lambda0} ,
  \swap{\mu}{\nu1} ] = [ \swap{\kappa}{\lambda1} , \swap{\mu}{\nu1} ]
= 1$ for all distinct $\kappa,\lambda,\mu,\nu \in X^{2}$ from
Relations~\ref{R:[23,23]}.  Consequently, in
the case of \eqstate{eq:[23,23]}, it remains to verify the relation in
the case when $x = y = 0$.  First observe that
\[
\swap{11}{011}^{\swap{000}{001}} =
\swap{11}{011}^{\swap{10}{000} \, \swap{10}{001} \, \swap{10}{000}}
= \swap{11}{011}
\]
by use of Equation~\eqref{eq:23^23=k01} and then repeated use of the
cases of Equation~\eqref{eq:[23,23]} that we already have; that is, $[
  \swap{11}{011} , \swap{000}{001} ] = 1$.  Now
\begin{align*}
  \swap{000}{001}^{\swap{10}{010}}
  &= \swap{000}{001}^{\swap{1}{01} \, \swap{11}{011}}
  &&\text{by Eq.~\eqref{eq:12-split}} \\
  &= \swap{000}{001}^{\swap{11}{011}}
  &&\text{by Eq.~\eqref{eq:[12,k01]}} \\
  &= \swap{000}{011}
  &&\text{as just established.}
\end{align*}
Thus, $\swap{10}{010}$~and~$\swap{000}{001}$ commute.  This means that
when we conjugate the relation $[ \swap{10}{010}, \swap{11}{001} ] =
1$, which is an instance of Equation~\eqref{eq:[23,23]} that we
already know, by the swap~$\swap{000}{001}$, we obtain
\[
[ \swap{10}{010} , \swap{11}{000} ] = 1,
\]
with use of Equation~\eqref{eq:23^k01=23}.  We now make use of
Lemma~\ref{lem:23^2} in our usual way to establish the missing case of
Equation~\eqref{eq:[23,23]}, namely when $x = y = 0$.

\subsection{\boldmath Relations involving $\Sw{12}$, $\Sw{13}$,
  $\Sw{2}$ and~$\Sw{23}$}

We now establish all the relations we require that involve swaps from
$\Sw{12}$,~$\Sw{13}$, $\Sw{2}$ and~$\Sw{23}$.  In view of the
relations that we have already obtained, we can assume that at least
one swap from~$\Sw{13}$ and at least one from~$\Sw{23}$ occur within
our relation.  The relations we need to establish are therefore, for
$x,y,z \in X$, the following:
\begin{align}
  [ \swap{x}{\bar{x}yz} , \swap{\bar{x}\bar{y}}{\bar{x}y\bar{z}} ] &=
  1
  \label{eq:[13,23]}
  \\
  \swap{x}{\bar{x}y}^{\swap{x}{\bar{x}\bar{y}z}} &=
  \swap{\bar{x}y}{\bar{x}\bar{y}z}
  \label{eq:12^13=23}
  \\
  \swap{x}{\bar{x}y}^{\swap{\bar{x}y}{\bar{x}\bar{y}z}} &=
  \swap{x}{\bar{x}\bar{y}z}
  \label{eq:12^23=13}
  \\
  \swap{x}{\bar{x}yz}^{\swap{x}{\bar{x}\bar{y}}} &=
  \swap{\bar{x}\bar{y}}{\bar{x}yz}
  \label{eq:13^12=23}
\end{align}

For \eqstate{eq:[13,23]}, take the equation $[ \swap{00}{10} ,
  \swap{01}{11} ] = 1$, conjugate by~$\swap{1}{00}$ and use the
definition in~\eqref{D:13} and
Lemma~\ref{lem:23help}\ref{(01,00x)-shrink} to conclude $[
  \swap{1}{000} , \swap{01}{001} ] = 1$.  The required equation then
follows, as usual, by use of Lemmas~\ref{lem:13^2}
and~\ref{lem:23^2}.

For \eqstate{eq:12^13=23}, we calculate as follows:
\begin{align*}
  \swap{1}{00}^{\swap{1}{01z}}
  &= \swap{1}{00}^{\swap{00}{01} \, \swap{1}{00z} \, \swap{00}{01}}
  &&\text{by the definition in~\eqref{D:13}} \\
  &= \swap{1}{01}^{\swap{1}{00z} \, \swap{00}{01}}
  &&\text{by the definition in~\eqref{D:12}} \\
  &= \swap{1}{01}^{\swap{1}{00} \, \swap{00}{1z} \, \swap{1}{00} \,
    \swap{00}{01}}
  &&\text{by the definition in~\eqref{D:13}} \\
  &= \swap{01}{1z}^{\swap{1}{00} \, \swap{00}{01}}
  &&\text{using Rels.\ \ref{R:12} and~\ref{R:S4}} \\
  &= \swap{01}{00z}^{\swap{00}{01}}
  &&\text{by Eq.~\eqref{eq:22^12=23}} \\
  &= \swap{00}{01z}
  &&\text{by Lem.~\ref{lem:23^2}.}
\end{align*}
The required equation now follows using Lemmas~\ref{lem:12^2},
\ref{lem:13^2} and~\ref{lem:23^2}.

For \eqstate{eq:12^23=13}, our main calculation is
\begin{multline*}
  \swap{1}{00}^{\swap{00}{01z}} =
  \swap{1}{00}^{\swap{1}{00} \, \swap{1}{01z} \, \swap{1}{00}} =
  \swap{1}{00}^{\swap{1}{01z} \, \swap{1}{00}} \\*
  = \swap{00}{01z}^{\swap{1}{00}} =
  \swap{1}{01z},
\end{multline*}
obtained by exploiting the definition of~$\swap{00}{01z}$
in~\eqref{D:23} and Equation \eqref{eq:12^13=23} above.  The required
equation again follows by Lemmas~\ref{lem:12^2}, \ref{lem:13^2}
and~\ref{lem:23^2}.

\eqstate{eq:13^12=23} follows from the definition of~$\swap{1}{01z}$
as in Equation~\eqref{D:13} with use of Lemmas~\ref{lem:12^2},
\ref{lem:13^2} and~\ref{lem:23^2}.

\subsection{\boldmath First relations involving $\Sw{2}$ and $\Sw{3}$
  only}

We now introduce swaps from~$\Sw{3}$ into the relations we verify.
The first step is to establish that swaps from~$\Sw{3}$ behave well
when we conjugate by one from~$\Sw{2}$ and then our other split
relation.  All other relations involving just swaps from
$\Sw{2}$~and~$\Sw{3}$ will have to wait until we have established some
more intermediate relations.  Accordingly, we start with the following
for $x,y \in X$, distinct $\kappa,\lambda \in X^{2}$ and $\tau \in
\Sw{2}$ for which Equation~\eqref{eq:33^2} is defined:
\begin{align}
  \swap{\kappa x}{\lambda y}^{\tau} &=
  \swap{(\kappa x)\!\cdot\!\tau\;}{(\lambda y)\!\cdot\!\tau}
  \label{eq:33^2}
  \\
  \swap{\kappa}{\lambda} &=
  \swap{\kappa0}{\kappa1} \, \swap{\lambda0}{\lambda1}
  \label{eq:22-split}
\end{align}

As with those from~$\Sw{23}$, we begin with a lemma to manipulate the
basic $\Sw{3}$\nbd swaps from the definition in~\eqref{D:33}.  In
particular, we will establish \eqstate{eq:33^2} as part~\ref{33^2} in
the lemma.

\begin{lemma}
  \label{lem:33help}
  \begin{enumerate}
  \item
    \label{(000,01x)-shrink}
    $\swap{000}{01x}^{\swap{1}{00}} = \swap{10}{01x}$, for any $x \in
    X$;
    
  \item
    \label{(001,011)-shrink}
    $\swap{001}{011}^{\swap{1}{00}} = \swap{11}{011}$;

  \item
    \label{33-commute}
    $\swap{000}{010}$,~$\swap{000}{011}$ and~$\swap{001}{011}$ each
    commute with~$\swap{10}{11}$.

  \item
    \label{33^2}
    All conjugacy relations of the form~$\sigma^{\tau} = \upsilon$,
    where $\sigma,\upsilon \in \Sw{3}$ and~$\tau \in \Sw{2}$, can be
    deduced from Relations~\ref{R:S4}--\ref{R:[33,23]}.

  \item
    \label{(001,010)-shrink}
    $\swap{001}{010}^{\swap{1}{00}} = \swap{11}{010}$;
  \end{enumerate}
\end{lemma}

\begin{proof}
  \ref{(000,01x)-shrink} We calculate as follows:
  \begin{align*}
    \swap{000}{01x}^{\swap{1}{00}} &=
    \swap{1}{000}^{\swap{1}{01x} \, \swap{1}{00}}
    &&\text{by the definition in~\eqref{D:33}} \\
    &= \swap{1}{000}^{\swap{1}{00} \, \swap{00}{01x}}
    &&\text{by Eq.~\eqref{eq:13^12=23}} \\
    &= \swap{00}{10}^{\swap{00}{01x}}
    &&\text{by the definition in~\eqref{D:13}} \\
    &= \swap{10}{01x}
    &&\text{by Eq.~\eqref{eq:23^23=22}.}
  \end{align*}

  Part~\ref{(001,011)-shrink} is established in exactly the same way.

  \ref{33-commute}~We perform the following calculations:
  \begin{align*}
    \lefteqn{\swap{000}{01x}^{\swap{10}{11} \, \swap{1}{00}}} \hspace*{3em} & \\*
    &= \swap{10}{01x}^{\swap{1}{00} \, \swap{10}{11} \, \swap{1}{00}}
    &&\text{by part~\ref{(000,01x)-shrink}} \\
    &= \swap{10}{01x}^{\swap{10}{000} \, \swap{11}{001} \,
      \swap{10}{11} \, \swap{1}{00}}
    &&\text{by Rel.~\ref{R:split}} \\
    &= \swap{11}{01x}^{\swap{11}{000} \, \swap{10}{001} \,
      \swap{1}{00}}
    &&\text{by Lem.~\ref{lem:23^2}} \\
    &= \swap{11}{01x}^{\swap{11}{000} \, \swap{1}{00}}
    &&\text{by Eq.~\eqref{eq:[23,23]} twice} \\
    &= \swap{11}{01x}^{\swap{11}{000} \, \swap{10}{000} \,
      \swap{11}{001}}
    &&\text{by Rel.~\ref{R:split}} \\
    &= \swap{11}{01x}^{\swap{10}{000} \, \swap{10}{11} \,
      \swap{11}{001}}
    &&\text{by Eq.~\eqref{eq:23^23=22}} \\
    &= \swap{11}{01x}^{\swap{10}{11} \, \swap{11}{001}}
    &&\text{by Eq.~\eqref{eq:[23,23]}} \\
    &= \swap{10}{01x}^{\swap{11}{001}}
    &&\text{by Lem.~\ref{lem:23^2}} \\
    &= \swap{10}{01x}
    &&\text{by Eq.~\eqref{eq:[23,23]}.}
  \end{align*}
  Thus $\swap{000}{01x}^{\swap{10}{11}} =
  \swap{10}{01x}^{\swap{1}{00}} = \swap{000}{01x}$, again by
  part~\ref{(000,01x)-shrink}.  A similar argument,
  using~\ref{(001,011)-shrink}, shows that $\swap{001}{011}$~commutes
  with~$\swap{10}{11}$.

  \ref{33^2}~When $\sigma$ (and~$\upsilon$) has the
  form~$\swap{\kappa0}{\kappa1}$, this result was established in
  Lemma~\ref{lem:k01help}\ref{k01^2}.  All remaining swaps in~$\Sw{3}$
  are defined by conjugating one of $\swap{000}{010}$,
  $\swap{000}{011}$ or~$\swap{001}{011}$ by some product of elements
  from~$\Sw{2}$.  The result then follows by the same argument, but
  now relying upon part~\ref{33-commute}.

  \ref{(001,010)-shrink}~appears to be similar to the first two parts
  of the lemma, but actually requires a different argument, based on
  what we have just established:
  \begin{multline*}
    \swap{001}{010}^{\swap{1}{00}} =
    \swap{000}{011}^{\swap{00}{01} \, \swap{1}{00}} =
    \swap{000}{011}^{\swap{1}{00} \, \swap{1}{01}} \\
    = \swap{10}{011}^{\swap{1}{01}} =
    \swap{11}{010},
  \end{multline*}
  by part~\ref{33^2}, Relation~\ref{R:12}, part~\ref{(000,01x)-shrink}
  and Equation~\eqref{eq:23^12=23}.
\end{proof}

\eqstate{eq:22-split} is now established by taking the equation
$\swap{1}{01} = \swap{10}{010} \, \swap{11}{011}$, which is an
instance of Equation~\eqref{eq:12-split}, and conjugating
by~$\swap{1}{00}$ to yield
\[
\swap{00}{01} = \swap{000}{010} \, \swap{001}{011},
\]
using Relation~\ref{R:12} and
Lemma~\ref{lem:33help}\ref{(000,01x)-shrink}
and~\ref{(001,011)-shrink}.  The required relation then follows by
conjugating by a product of elements from~$\Sw{2}$ and using the
Relations~\ref{R:S4} and Lemma~\ref{lem:33help}\ref{33^2}.

\subsection{\boldmath Further intermediate relations involving
  $\Sw{23}$ and~$\Sw{3}$}

Our principal direction of travel at this stage is to complete the
verification of those relations involving swaps from
$\Sw{2}$~and~$\Sw{3}$ to supplement those in Equations~\eqref{eq:33^2}
and~\eqref{eq:22-split}.  However, to achieve this we need some
intermediate results making use of swaps from~$\Sw{23}$, specifically,
for $x \in X$, distinct $\kappa,\lambda,\mu \in X^{2}$ and distinct
$\alpha,\beta \in X^{3}$ for which $\kappa \perp \alpha,\beta$, the
following:
\begin{align}
  \swap{\kappa}{\alpha}^{\swap{\kappa}{\beta}} &= \swap{\alpha}{\beta}
  \label{eq:23^23=33}
  \\
  \swap{\kappa}{\alpha}^{\swap{\alpha}{\beta}} &= \swap{\kappa}{\beta}
  \label{eq:23^33=23}
  \\
  [ \swap{\kappa}{\lambda x}, \swap{\mu0}{\mu1} ] &= 1
  \label{eq:[23,k01]}
\end{align}

For \eqstate{eq:23^23=33}, first note that
Equation~\eqref{eq:23^23=k01} deals with the case when
$\alpha$~and~$\beta$ share the same two-letter prefix.  For the
remaining cases, first observe
\begin{align*}
  \lefteqn{\swap{10}{000}^{\swap{10}{010} \, \swap{1}{010}}}\hspace*{4em} & \\*
  &= \swap{10}{000}^{\swap{10}{010} \, \swap{11}{011} \,
    \swap{1}{010}}
  &&\text{by Eqs.~\eqref{eq:[23-k0,23-k1]} and \eqref{eq:[23,23]}} \\
  &= \swap{10}{000}^{\swap{1}{01} \, \swap{1}{010}}
  &&\text{by Eq.~\eqref{eq:12-split}} \\
  &= \swap{10}{000}^{\swap{01}{10} \, \swap{1}{01}}
  &&\text{by Eq.~\eqref{eq:13^12=22}} \\
  &= \swap{01}{000}^{\swap{1}{01}}
  &&\text{by Lem.~\ref{lem:23^2}} \\
  &= \swap{1}{000}
  &&\text{by Eq.~\eqref{eq:13^12=23}.}
\end{align*}
Hence $\swap{10}{000}^{\swap{10}{010}} = \swap{1}{000}^{\swap{1}{010}}
= \swap{000}{010}$, using the definition in Equation~\eqref{D:33}.  A
similar argument establishes $\swap{11}{000}^{\swap{11}{011}} =
\swap{1}{000}^{\swap{1}{011}} = \swap{000}{011}$.  Equally we apply a
variant of the argument to obtain further equations:
\begin{align*}
  \swap{10}{011}^{\swap{10}{000} \, \swap{1}{011}}
  &= \swap{10}{011}^{\swap{1}{00} \, \swap{1}{011}}
  &&\text{arguing as before} \\
  &= \swap{10}{011}^{\swap{00}{011} \, \swap{1}{00}}
  &&\text{by Eq.~\eqref{eq:13^12=23}} \\
  &= \swap{00}{10}^{\swap{1}{00}}
  &&\text{by Eq.~\eqref{eq:23^23=22}} \\
  &= \swap{1}{000}
  &&\text{by the definition in~\eqref{D:13}.}
\end{align*}
Thus we obtain $\swap{10}{011}^{\swap{10}{000}} =
\swap{1}{000}^{\swap{1}{011}} = \swap{000}{011}$.  Similarly we
determine the equation $\swap{11}{011}^{\swap{11}{001}} =
\swap{001}{011}$.  Thus given any choice of $x,y \in X$, we have
obtained one example of a relation
\[
\swap{\kappa}{\lambda x}^{\swap{\kappa}{\mu y}} = \swap{\lambda x}{\mu y}
\]
(for some choice of distinct $\kappa$,~$\lambda$ and~$\mu$.)  We can
then obtain \emph{all} examples by use of Lemmas~\ref{lem:23^2}
and~\ref{lem:33help}\ref{33^2}.  This completes the establishment of
Equation~\eqref{eq:23^23=33}.

Equation~\eqref{eq:23^k01=23} is already \eqstate{eq:23^33=23} in the
case when $\alpha$~and~$\beta$ share the same two-letter prefix.  For
the remaining cases, use Equation~\eqref{eq:23^23=33} to tell us
$\swap{1x}{000}^{\swap{1x}{01y}} = \swap{000}{01y}$ for any $x,y \in
X$.  Conjugate this equation by~$\swap{1}{00}$ and use
Equation~\eqref{eq:23^12=23} and parts of Lemma~\ref{lem:33help} to
establish $\swap{10}{00x}^{\swap{00x}{01y}} = \swap{10}{01y}$.  All
cases of Equation~\ref{eq:23^33=23} now follow by conjugating by an
appropriate product of elements from~$\Sw{2}$.

For \eqstate{eq:[23,k01]}, start with the fact that
$\swap{10}{11}$~commutes with $\swap{000}{01x}$ for any~$x \in X$
(Lemma~\ref{lem:33help}\ref{33-commute}).  Conjugate by~$\swap{1}{00}$
and use Lemmas~\ref{lem:k01help}\ref{(000,001)-shrink}
and~\ref{lem:33help}\ref{(000,01x)-shrink} to conclude $[
  \swap{000}{001} , \swap{01}{01x} ] = 1$.  Then conjugating by a
product of swaps from~$\Sw{2}$ establishes the required equation.

\subsection{\boldmath Remaining relations involving only $\Sw{2}$,
  $\Sw{23}$ and~$\Sw{3}$}

Having established the intermediate relations, we can now establish
all remaining relations involving swaps only from
$\Sw{2}$~and~$\Sw{3}$.  We obtain those also involving swaps
from~$\Sw{23}$ at the same time.  When one analyses the relations
required, we find that they are, for $x,y \in X$, distinct
$\kappa,\lambda \in X^{2}$ and distinct $\alpha,\beta,\gamma,\delta
\in X^{3}$, the following:
\begin{align}
  [ \swap{\kappa0}{\kappa1} , \swap{\lambda x}{\mu y} ] &= 1
  \label{eq:[k01,33]}  
  \\
  \swap{\kappa x}{\lambda y}^{\swap{\kappa0}{\kappa1}} &=
  \swap{\kappa\bar{x}}{\lambda y}
  \label{eq:33^k01}
  \\
  [ \swap{\kappa}{\alpha} , \swap{\beta}{\gamma} ] &= 1
  &&\text{for $\kappa \perp \alpha,\beta,\gamma$}
  \label{eq:[23,33]}
  \\
  \swap{\alpha}{\beta}^{\swap{\alpha}{\gamma}} &= \swap{\beta}{\gamma}
  \label{eq:33^33}
  \\
  [ \swap{\alpha}{\beta} , \swap{\gamma}{\delta} ] &= 1
  \label{eq:[33,33]}
\end{align}

We establish \eqstate{eq:[k01,33]} by choosing $\nu$~to be the element
in~$X^{2} \setminus \{\kappa,\lambda,\mu\}$, using
Equation~\eqref{eq:23^23=33} to tell us $\swap{\lambda x}{\mu y} =
\swap{\nu}{\lambda x}^{\swap{\nu}{\mu y}}$ and then, as
Equation~\eqref{eq:[23,k01]} says that both $\swap{\nu}{\lambda x}$
and~$\swap{\nu}{\mu y}$ commute with~$\swap{\kappa0}{\kappa1}$, we
establish Equation~\eqref{eq:[k01,33]}.

By Lemma~\ref{lem:23^2}, $\swap{1x}{01y}^{\swap{10}{11}} =
\swap{1\bar{x}}{01y}$ for any $x,y \in X$.  Conjugate
by~$\swap{1}{00}$ and use
Lemma~\ref{lem:k01help}\ref{(000,001)-shrink} and parts of
Lemma~\ref{lem:33help} to conclude $\swap{00x}{01y}^{\swap{000}{001}}
= \swap{00\bar{x}}{01y}$.  \eqstate{eq:33^k01} then follows.

Our final equation involving swaps from $\Sw{23}$~and~$\Sw{3}$ is
\eqstate{eq:[23,33]}.  We need to establish this in a number of
stages.  If $x \in X$, we know that $\swap{10}{11}$~commutes with the
swap~$\swap{000}{01x}$ by Lemma~\ref{lem:33help}\ref{33-commute}.  If
we conjugate by~$\swap{1}{00}$ and use
Lemma~\ref{lem:k01help}\ref{(000,001)-shrink} and
Lemma~\ref{lem:33help}\ref{(000,01x)-shrink}, we conclude $[
  \swap{10}{01x}, \swap{000}{001} ] = 1$.  We then deduce
Equation~\eqref{eq:[23,33]} in the case when $\beta$~and~$\gamma$
share the same two-letter prefix in our now established manner.

The second case of the equation is when $\alpha$~and~$\beta$ share
their two-letter prefix.  Equation~\eqref{eq:[23,23]} tells us that $[
  \swap{1x}{000} , \swap{1\bar{x}}{01y} ] = 1$ for any $x,y \in X$.
Now conjugate by~$\swap{1}{00}$ and use Equation~\eqref{eq:23^12=23}
and parts of Lemma~\ref{lem:33help} to conclude $[ \swap{10}{00x} ,
  \swap{00\bar{x}}{01y} ] = 1$.  Equation~\eqref{eq:[23,33]} when
$\alpha$~and~$\beta$ share their two-letter prefix then follows.

It remains to deal with the case when $\alpha$,~$\beta$ and~$\gamma$
have distinct two-letter prefixes.  One particular case is our
Relation~\ref{R:[33,23]}: $[\swap{10}{110} , \swap{000}{010} ] = 1$.
Conjugating by~$\swap{110}{111}$, using Equation~\eqref{eq:23^k01=23}
and~\eqref{eq:[k01,33]}, we deduce $[\swap{10}{111} , \swap{000}{010}
] = 1$.  Similarly, conjugating what we now have, $[ \swap{10}{11x} ,
\swap{000}{010} ] = 1$ for any $x \in X$, by~$\swap{000}{001}$ and
use~\eqref{eq:[23,k01]} and~\eqref{eq:33^k01} to now conclude that
$[ \swap{10}{11x} , \swap{00y}{010}] = 1$ for all $x,y \in X$.
Finally use the same argument, conjugating by~$\swap{010}{011}$, to
conclude $[ \swap{10}{11x} , \swap{00y}{01z} ] = 1$ for all $x,y,z \in
X$.  The remaining case of Equation~\eqref{eq:[23,33]} now follows.

One case of \eqstate{eq:33^33}, namely when $\alpha$~and~$\gamma$
share the same two-letter prefix, has already been established as
Equation~\eqref{eq:33^k01}.  For the case when $\alpha$~and~$\beta$
share the same two-letter prefix, start with the equation
$\swap{10}{11}^{\swap{1x}{01y}} = \swap{1\bar{x}}{01y}$, for $x,y \in
X$, as given by Equation~\eqref{eq:23^23=22}.  Conjugate
by~$\swap{1}{00}$ and use
Lemma~\ref{lem:k01help}\ref{(000,001)-shrink} and parts from
Lemma~\ref{lem:33help} to conclude $\swap{000}{001}^{\swap{00x}{01y}}
= \swap{00\bar{x}}{01y}$.  From this the general formula
$\swap{\kappa0}{\kappa1}^{\swap{\kappa x}{\lambda y}} =
\swap{\kappa\bar{x}}{\lambda y}$ follows for distinct $\kappa,\lambda
\in X^{2}$.

For the case when $\alpha$,~$\beta$ and~$\gamma$ have distinct
two-letter prefixes, say $\alpha = \kappa x$, $\beta = \lambda y$ and
$\gamma = \mu z$, choose $\nu$~to be the other element of~$X^{2}$.
Then
\begin{align*}
\swap{\alpha}{\beta}^{\swap{\alpha}{\gamma}} =
\swap{\kappa x}{\lambda y}^{\swap{\kappa x}{\mu z}} &=
\swap{\kappa x}{\lambda y}^{\swap{\nu}{\kappa x} \, \swap{\nu}{\mu z}
  \, \swap{\nu}{\kappa x}} \\
&=
\swap{\nu}{\lambda y}^{\swap{\nu}{\mu z} \, \swap{\nu}{\kappa x}} \\
&= \swap{\lambda y}{\mu z}^{\swap{\nu}{\kappa x}} =
\swap{\lambda y}{\mu z} =
\swap{\beta}{\gamma},
\end{align*}
by Equations~\eqref{eq:23^23=33} (used three times)
and~\eqref{eq:[23,33]}.

Part of \eqstate{eq:[33,33]} has already been established as
Equation \eqref{eq:[k01,33]}, but we shall deal with our required
relation in full generality.  Indeed, first assume that
$\alpha$,~$\beta$, $\gamma$ and~$\delta$ have between them at most
three distinct two-letter prefixes.  Let $\nu \in X^{2}$ be different
from those two-letter prefixes.  Then write $\swap{\gamma}{\delta}$ as
$\swap{\nu}{\gamma} \, \swap{\nu}{\delta} \, \swap{\nu}{\gamma}$, by
Equation~\eqref{eq:23^23=33}, and observe this commutes
with~$\swap{\alpha}{\beta}$ using Equation~\eqref{eq:[23,33]}.

The case when $\alpha$,~$\beta$, $\gamma$ and~$\delta$ have four
distinct two-letter prefixes can then be deduced as follows.  Suppose
as $\alpha = \kappa x$ for some $\kappa \in X^{2}$ and $x \in X$.  By
the previous case, $\swap{\alpha}{\beta}$~commutes with both
$\swap{\kappa\bar{x}}{\gamma}$ and~$\swap{\kappa\bar{x}}{\delta}$, and
hence it also commutes with $\swap{\gamma}{\delta} =
\swap{\kappa\bar{x}}{\gamma}^{\swap{\kappa\bar{x}}{\delta}}$, using
Equation~\eqref{eq:33^33}.

\subsubsection{\boldmath Relations involving~$\Sw{3}$, at least one of
  $\Sw{12}$, $\Sw{13}$ and~$\Sw{23}$, and possibly~$\Sw{2}$}

We now establish the final batch of relations for this section.  These
involve swaps from~$\Sw{3}$ and at least one of $\Sw{12}$,~$\Sw{13}$
and~$\Sw{23}$, and are the following for $x,y,z,t \in X$ and distinct
$\alpha,\beta,\gamma \in X^{3}$ satisfying $x \not\prec
\alpha,\beta,\gamma$:
\begin{align}
  \swap{\bar{x}y0}{\bar{x}y1}^{\swap{x}{\bar{x}y}} &= \swap{x0}{x1}
  \label{eq:k01^12=22}
  \\
  [ \swap{x}{\alpha} , \swap{\beta}{\gamma} ] &= 1
  \label{eq:[13,33]}
  \\
  \swap{x}{\alpha}^{\swap{\alpha}{\beta}} &= \swap{x}{\beta}
  \label{eq:13^33=13}
  \\
  \swap{\alpha}{\beta}^{\swap{x}{\alpha}} &= \swap{x}{\beta}
  \label{eq:33^13=13}
  \\
  \swap{\bar{x}yz}{\bar{x}\bar{y}t}^{\swap{x}{\bar{x}y}} &=
  \swap{xz}{\bar{x}\bar{y}t}
  \label{eq:33^12=23}
\end{align}

\eqstate{eq:k01^12=22} follows from
Lemma~\ref{lem:k01help}\ref{(000,001)-shrink} by our standard
$\Sw{2}$-conjug\-ation argument.

\eqstate{eq:[13,33]} follows by taking the relation $[ \swap{00}{1y} ,
  \swap{1\bar{y}}{01z} ] = 1$ and the relation $[ \swap{00}{1y} ,
  \swap{010}{011} ] = 1$, which hold by Lemmas~\ref{lem:23^2}
and~\ref{lem:33help}\ref{33^2} respectively, and then conjugating
by~$\swap{1}{00}$ and proceeding as in previous arguments to conclude that
$[ \swap{x}{\kappa y} , \swap{\kappa\bar{y}}{\lambda z} ] = 1$ and
$[ \swap{x}{\kappa y} , \swap{\lambda0}{\lambda1} ] = 1$ for any
$x,y,z \in X$ and any distinct $\kappa,\lambda \in X^{2}$ with $x
\not\prec \kappa,\lambda$.

To establish \eqstate{eq:13^33=13}, conjugate the already established
equations $\swap{00}{1y}^{\swap{10}{11}} = \swap{00}{1\bar{y}}$ and
$\swap{00}{1y}^{\swap{1y}{01z}} = \swap{00}{01z}$, for $y,z \in X$,
by~$\swap{1}{00}$.  This yields $\swap{1}{00y}^{\swap{000}{001}} =
\swap{1}{00\bar{y}}$ and $\swap{1}{00y}^{\swap{00y}{01z}} =
\swap{1}{01z}$, which now yields the two forms of
Equation~\eqref{eq:13^33=13}: \ $\swap{x}{\kappa
  y}^{\swap{\kappa0}{\kappa1}} = \swap{x}{\kappa\bar{y}}$ and
$\swap{x}{\kappa y}^{\swap{\kappa y}{\lambda z}} = \swap{x}{\lambda
  z}$ for any $x,y,z \in X$ and distinct $\kappa,\lambda \in X^{2}$
with $x \not\prec \kappa,\lambda$.

For \eqstate{eq:33^13=13}, we shall show
$\swap{000}{001}^{\swap{1}{00x}} = \swap{1}{00\bar{x}}$ and
$\swap{00x}{01y}^{\swap{1}{01y}} = \swap{1}{00x}$ for any $x,y \in X$.
Four of these occurrences are found in the definitions
in~\eqref{D:33}, while the other two are deduced by conjugating
$\swap{10}{11}^{\swap{00}{10}} = \swap{00}{11}$ and
$\swap{11}{010}^{\swap{00}{010}} = \swap{00}{11}$ by~$\swap{1}{00}$.
The required equation then follows.

Finally, for \eqstate{eq:33^12=23}, from
Lemma~\ref{lem:33help}: $\swap{00z}{01t}^{\swap{1}{00}} =
\swap{1z}{01t}$ for any $z,t \in X$.  Then as in the previous
equations we conjugate by products of elements from~$\Sw{2}$.

\spc

We have now established all required relations of the form
$\sigma^{\tau} = \upsilon$ where $\sigma$,~$\tau$ and~$\upsilon$ come
from the sets $\Sw{2}$,~$\Sw{12}$, $\Sw{13}$, $\Sw{23}$ and~$\Sw{3}$.
This is the first stage in establishing the existence of the
homomorphism~$\psi$ in Diagram~\eqref{eq:hexagon} and, in particular,
the verification of Theorem~\ref{thm:main}.

\section{Verifying the Cannon--Floyd--Parry relations}
\label{sec:CFP}

In this section we describe how to verify that all the relations that
hold in R.~Thompson's group~$V$ can be deduced from those assumed in
Relations~\ref{R:S4}--\ref{R:[33,23]}.  We shall rely upon the work in
the previous section.  One might wonder whether it is possible to
proceed more directly to show, for example, that all relations holding
in~$V$ can be deduced from the infinitely many in the presentation in
Theorem~\ref{thm:infpres}.  It is a consequence of our results that
this can be done, but our own attempt to do so resulted in overly long
arguments replicating those already found in Section~6 of~\cite{CFP}.
We have chosen the more direct method of verifying the finite set of
relations known already to define~$V$.

In their paper (see~\cite[Lemma~6.1]{CFP}), Cannon--Floyd--Parry
provide the following presentation for~$V$.  It has generators
$A$,~$B$, $C$ and~$\pi_{0}$ and relations
\begin{center}
  \renewcommand{\theenumi}{CFP\arabic{enumi}}
  \begin{tabular}{rccrc}
    CFP1. & $[AB^{-1}, X_{2}] = 1$; &\qquad\qquad
    &CFP8. & $\pi_{1} \pi_{3} = \pi_{3} \pi_{1}$; \\
    CFP2. & $[AB^{-1}, X_{3}] = 1$; &
    &CFP9. & $(\pi_{2} \pi_{1})^{3} = 1$; \\
    CFP3. & $C_{1} = BC_{2}$; &
    &CFP10. & $X_{3} \pi_{1} = \pi_{1} X_{3}$; \\
    CFP4. & $C_{2}X_{2} = BC_{3}$; &
    &CFP11. & $\pi_{1} X_{2} = B \pi_{2} \pi_{1}$; \\
    CFP5. & $C_{1}A = C_{2}^{2}$; &
    &CFP12. & $\pi_{2} B = B \pi_{3}$; \\
    CFP6. & $C_{1}^{3} = 1$; &
    &CFP13. & $\pi_{1} C_{3} = C_{3} \pi_{2}$; \\
    CFP7. & $\pi_{1}^{2} = 1$; &
    &CFP14. & $(\pi_{1} C_{2})^{3} = 1$;
  \end{tabular}
\end{center}
where the elements appearing here are defined by the following
formulae $C_{n} = A^{-n+1} C B^{n-1}$, \ $X_{n} = A^{-n+1} B A^{n-1}$
both for $n \geq 1$, \ $\pi_{1} = C_{2}^{-1} \pi_{0} C_{2}$ and
$\pi_{n} = A^{-n+1} \pi_{1} A^{n-1}$ for $n \geq 2$.

Recall from Section~\ref{sec:prelims} that $P_{3}$~is the group
presented by generators $a = \swap{00}{01}$, \ $b = \swap{01}{10} \,
\swap{01}{11}$ and $c = \swap{1}{00}$ subject to relations
\ref{R:S4}--\ref{R:[33,23]}.  Define four new elements of~$P_{3}$ by
\begin{align*}
  \bar{A} &= \swap{0}{1} \, \swap{0}{10} \, \swap{10}{11};
  &\bar{B} &= \swap{10}{11} \, \swap{10}{110} \, \swap{110}{111}; \\
  \bar{C} &= \swap{10}{11} \, \swap{0}{10};
  &\bar{\pi}_{0} &= \swap{0}{10},
\end{align*}
and then new elements~$\bar{C}_{n}$, $\bar{X}_{n}$ and~$\bar{\pi}_{n}$
for $n \geq 1$ defined in terms of these four by same formulae used
when defining the relations for~$V$.

It is a consequence of the relations involving swaps from the sets
$\Sw{2}$,~$\Sw{12}$, $\Sw{13}$, $\Sw{23}$ and~$\Sw{3}$ established in
the previous section (i.e.,
Equations~\eqref{eq:12^11}--\eqref{eq:33^12=23}) that the elements
$\bar{A}$,~$\bar{B}$, $\bar{C}$ and $\bar{\pi}_{0}$ satisfy
CFP1--CFP14\@.  To verify this is a sequence of calculations.  Below
we present the verification of CFP1 for these elements.  For the
entertainment of the reader, Relation~CFP2 required the longest
calculation whilst the others are more straightforward.  The following
formulae are useful for this work.

\begin{lemma}
  The following formulae hold in~$G$:
  \begin{enumerate}
  \item
    \label{AB^-1}
    $\bar{A} \bar{B}^{-1} = \swap{00}{01} \, \swap{01}{10} \,
    \swap{0}{10}$;
  \item
    \label{X2}
    $\bar{X}_{2} = \swap{0}{11} \, \swap{00}{01} \, \swap{00}{010} \,
    \swap{010}{011} \, \swap{0}{11}$;
  \item
    \label{X3}
    $\bar{X}_{3} = \swap{0}{111} \, \swap{00}{01} \, \swap{00}{010} \,
    \swap{010}{011} \, \swap{0}{111}$;
  \item
    \label{C2}
    $\bar{C}_{2} = \swap{0}{10} \, \swap{0}{111} \, \swap{110}{111}$;
  \item $\bar{C}_{3} = \swap{0}{110} \, \swap{10}{111} \,
    \swap{0}{100} \, \swap{0}{101} \, \swap{10}{110} \,
    \swap{110}{111}$;
  \item $\bar{\pi}_{1} = \swap{10}{110}$;
  \item $\bar{\pi}_{2} = \swap{0}{11} \, \swap{00}{010} \,
    \swap{0}{11}$;
  \item $\bar{\pi}_{3} = \swap{0}{111} \, \swap{00}{010} \,
    \swap{0}{111}$.
  \end{enumerate}
\end{lemma}

\begin{proof}
  We verify the two formulae, (i)~and~(ii), required to verify CFP1.
  Below, we principally rely upon the conjugacy
  relations, although a split relation is applied in one step.
  The other formulae listed are established
  similarly.
  
  \ref{AB^-1}~We calculate
  \begin{align*}
    \bar{A} \bar{B}^{-1} &=
    \swap{0}{1} \, \swap{0}{10} \, \swap{10}{11} \cdot \swap{110}{111}
    \, \swap{10}{110} \, \swap{10}{11} \\
    &= \swap{0}{1} \, \swap{0}{10} \, \swap{100}{101} \,
    \swap{11}{100} \\
    &= \swap{0}{1} \, \swap{00}{01} \, \swap{00}{11} \, \swap{0}{10}
    \\
    &= \swap{00}{10} \, \swap{01}{11} \, \swap{00}{01} \,
    \swap{00}{11} \, \swap{0}{10} \\
    &= \swap{00}{01} \, \swap{01}{10} \, \swap{0}{10}.
  \end{align*}

  \ref{X2}~We start with the definition of $\bar{A}$ and~$\bar{B}$:
  \begin{align*}
    \bar{X}_{2} &= \bar{A}^{-1} \bar{B} \bar{A} \\
    &= \swap{10}{11} \, \swap{0}{10} \, \swap{0}{1} \cdot
    \swap{10}{11} \, \swap{10}{110} \, \swap{110}{111} \cdot
    \swap{0}{1} \, \swap{0}{10} \, \swap{10}{11} \\
    &= \swap{10}{11} \, \swap{0}{10} \, \swap{00}{01} \,
    \swap{00}{010} \, \swap{010}{011} \, \swap{0}{10} \, \swap{10}{11}
    \\
    &= \swap{0}{11} \, \swap{00}{01} \, \swap{00}{010} \,
    \swap{010}{011} \, \swap{0}{11}.
  \end{align*}
\end{proof}

We now verify that the elements $\bar{A}$,~$\bar{B}$, $\bar{C}$
and~$\bar{\pi}_{0}$ of our~$P_{3}$ satisfy the relation~CFP1:
\begin{align*}
  %%%%%%
  % CFP1
  [ \bar{A}\bar{B}^{-1} , \bar{X}_{2} ]
  &= \swap{0}{10} \, \swap{01}{10} \, \swap{00}{01} \cdot \swap{0}{11}
  \, \swap{010}{011} \, \swap{00}{010} \, \swap{00}{01} \,
  \swap{0}{11} \\*
  &\qquad \cdot \swap{00}{01} \, \swap{01}{10} \, \swap{0}{10} \cdot
  \swap{0}{11} \, \swap{00}{01} \, \swap{00}{010} \, \swap{010}{011}
  \\*
  &\qquad \cdot \swap{0}{11} \\
  &= \swap{0}{11} \, \swap{10}{11} \, \swap{10}{111} \,
  \swap{110}{111} \, \swap{010}{011} \, \swap{00}{010} \,
  \swap{00}{01} \\*
  &\qquad \cdot \swap{0}{11} \, \swap{00}{01} \, \swap{01}{10} \,
  \swap{0}{10} \, \swap{0}{11} \, \swap{00}{01} \, \swap{00}{010} \\*
  &\qquad \cdot \swap{010}{011} \, \swap{0}{11} \\
  &= \swap{0}{11} \, \swap{10}{11} \, \swap{10}{111} \,
  \swap{110}{111} \, \swap{010}{011} \, \swap{00}{010} \,
  \swap{00}{01} \\*
  &\qquad \cdot \swap{110}{111} \, \swap{10}{111} \, \swap{10}{11} \,
  \swap{00}{01} \, \swap{00}{010} \, \swap{010}{011} \, \swap{0}{11}
  \\
  &= \swap{0}{11} \, \swap{11}{101} \, \swap{100}{101} \,
  \swap{010}{011} \, \swap{00}{010} \, \swap{00}{01} \,
  \swap{100}{101} \\*
  &\qquad \cdot \swap{11}{101} \, \swap{00}{01} \, \swap{00}{010} \,
  \swap{010}{011} \, \swap{0}{11} \\
  &= \swap{0}{11} \, \swap{11}{101} \, \swap{100}{101} \,
  \swap{010}{011} \, \swap{00}{010} \, \swap{100}{101} \\*
  &\qquad \cdot \swap{11}{101} \, \swap{00}{010} \, \swap{010}{011} \,
  \swap{0}{11} \\
  &= \swap{0}{11} \, \swap{11}{101} \, \swap{010}{011} \,
  \swap{00}{010} \, \swap{11}{101} \, \swap{00}{010} \,
  \swap{010}{011} \\*
  &\qquad \cdot \swap{0}{11} \\
  &= \swap{0}{11} \, \swap{11}{101} \, \swap{010}{011} \,
  \swap{11}{101} \, \swap{010}{011} \, \swap{0}{11} \\
  &= 1
\end{align*}
(by first collecting $\swap{0}{11}$ to the left, then conjugating some
swaps by~$\swap{0}{11}$, some by~$\swap{10}{11}$, some
by~$\swap{00}{01}$, some by~$\swap{100}{101}$, then single swaps
by~$\swap{00}{010}$ and by~$\swap{11}{101}$, and finally exploiting
the fact our swaps have order~$2$).

Once we have established the fourteen relations CFP1--CFP14, it
follows that there is indeed a surjective homomorphism $\psi \colon V
\to \langle \bar{A}, \bar{B}, \bar{C}, \bar{\pi}_{0} \rangle$, as
indicated in the Diagram~\eqref{eq:hexagon} and used in the proof of
Theorems~\ref{thm:infpres} and~\ref{thm:main}.

\section{Final details for the proofs}
\label{sec:finaldetails}

In this section, we complete the technical details relied upon in the
proofs given in Section~\ref{sec:prelims}.

We first need to show that the subgroup $\langle \bar{A}, \bar{B},
\bar{C}, \bar{\pi}_{0} \rangle$ coincides with the group~$P_{3}$.
Indeed observe this subgroup contains all the following elements:
\begin{align*}
  \bar{\pi}_{0} &= \swap{1}{00} = c \\
  \bar{C} \bar{\pi}_{0} &= \swap{10}{11} \\
  ( \bar{C} \bar{\pi}_{0} )^{\bar{A} \bar{C}} &=
  \swap{10}{11}^{\swap{0}{1}} = \swap{00}{01} = a \\
  \bar{\pi}_{0}^{\bar{B}^{-1} \bar{C}} &=
  \swap{0}{10}^{\swap{110}{111} \, \swap{10}{110} \, \swap{0}{10}} =
  \swap{10}{110} \\
  \bar{\pi}_{0}^{\bar{B}^{-1} \bar{C}} \bar{C} \bar{\pi}_{0} \bar{B}
  &= \swap{110}{111}
\end{align*}
and
\[
\swap{10}{110}^{\swap{110}{111} \, \swap{10}{11} \, \swap{0}{10} \,
  \swap{10}{11}} = \swap{01}{10}.
\]
The Relations~\ref{R:S4} ensure that $b \in \langle \swap{00}{01} ,
\swap{01}{10} , \swap{10}{11} \rangle$, and so we now conclude that
this subgroup generated by $\bar{A}$,~$\bar{B}$, $\bar{C}$
and~$\bar{\pi}_{0}$ is the whole group~$P_{3}$.  This establishes the
following result and means that we have now completed all the details
required for the proofs of Theorem~\ref{thm:infpres}
and~\ref{thm:main}.

\begin{prop}
  \label{prop:P3-gens}
  The elements $\bar{A}$,~$\bar{B}$, $\bar{C}$ and~$\bar{\pi}_{0}$,
  defined earlier, generate the group~$P_{3}$. \qed
\end{prop}

Finally, we deduce a $2$\nbd generator presentation for~$V$ from
Theorem~\ref{thm:main}.  As noted in Section~\ref{sec:prelims}, the
following is the intermediate step used to establish
Theorem~\ref{thm:2gen-KB}.  Although the latter depends on computer
calculation, the following is established by purely theoretical
methods in line with our proofs of Theorems~\ref{thm:infpres}
and~\ref{thm:main}.

\begin{cor}
  \label{cor:2gen-Tietze}
  R.~Thompson's group~$V$ has a finite presentation with two
  generators and nine relations.
\end{cor}

\begin{proof}
  We work in the group~$P_{3}$.  Define $u$~and~$v$ to be the elements
  \[
  u = \swap{00}{01} \, \swap{10}{110} \, \swap{10}{111}
  \qquad \text{and} \qquad
  v = b = ( 01 \; 10 \; ).
  \]
  (When interpreted via the isomorphism from~$P_{3}$ to~$V$, obtained
  by composing the maps specified in the diagram~\eqref{eq:hexagon},
  these two elements correspond to the element of~$V$ given by
  tree-pairs in Figure~\ref{fig:2gens} in the Introduction.)

  If we rely upon the relations that hold in~$P_{3}$ (i.e., simply
  calculating within R.~Thompson's group~$V$, as, by this stage, we
  have completed all steps in establishing Theorem~\ref{thm:main}),
  then we can obtain a formula for~$u$ as a product of the generators
  $a$,~$b$ and~$c$, for example,
  \[
  u = w(a,b,c) = a(a^ba^{b^{-1}})^{caca^{b}a^{b^{-1}a}}
  \]
  and the following formulae:
  \begin{align*}
    a &= \swap{00}{01} = u^{3}, \\
    \swap{10}{000} &= (u^{3})^{vu^{-2}vu^{3}}, \\
    \swap{11}{001} &= (u^{3})^{vu^{-1}vu^{3}v}
  \end{align*}
  Therefore $c = \gamma(u,v) = (u^{3})^{vu^{-2}vu^{3}} \,
  (u^{3})^{vu^{-1}vu^{3}v}$.  If $r_{1}(a,b,c)$,~$r_{2}(a,b,c)$,
  \dots, $r_{8}(a,b,c)$ denote the words in $a$,~$b$ and~$c$ that
  define our relations (see the Equations~\eqref{eq:KB-words} in
  Section~\ref{sec:prelims}, or alternatively,
  Equations~\eqref{eq:R-words}), then applying Tietze transformations
  shows that
  \begin{multline*}
    V \cong
    \left\langle \, u,v \; \middle| \; r_{1}\bigl(u^{3},v,\gamma(u,v)\bigr) =
    r_{2}\bigl(u^{3},v,\gamma(u,v)\bigr) = \dots = \right. \\
    \left. r_{8}\bigl(u^{3},v,\gamma(u,v)\bigr) = 1, 
    \quad u = w\bigl(u^{3},v,\gamma(u,v)\bigr) \, \right\rangle.
  \end{multline*}
  This establishes the corollary.
\end{proof}

As we described in Section~\ref{sec:prelims},
Theorem~\ref{thm:2gen-KB} is deduced from this presentation.  We
produce the formulae $r_{i}\bigl(u^{3},v,\gamma(u,v)\bigr)$, for $i =
1$,~$2$, \dots,~$8$, by taking $r_{i}$~to be the formulae
in~\eqref{eq:KB-words}.  We then apply the process of producing
equivalent relations for this $2$\nbd generator presentation by
repeatedly using the Knuth--Bendix Algorithm as described in
Section~\ref{sec:prelims}.  It is during this process that we discover
that two of the relations can be omitted since the relevant word
reduces to the identity.  The remaining seven relations reduce to
those listed in Theorem~\ref{thm:2gen-KB}.

\paragraph{Acknowledgements:} The authors wish to acknowledge partial
support by EPSRC grant EP/H011978/1.  We further wish to recognise our
gratitude to the authors of the computer packages that greatly
assisted when we performed calculations and also helped to reduce our
presentations: \GAP, KBMAG and the Java VTrees applet written by the
first author and Roman Kogan.

\end{document}